[1994/12/01]
\documentclass[10pt, reqno]{amsart}
\usepackage[english, activeacute]{babel}
\usepackage{amsmath,amsthm,amsxtra}
\usepackage{epsfig}
\usepackage{amssymb}
\usepackage{latexsym}
\usepackage{amsfonts}
\usepackage{hyperref}
\usepackage{pdftricks}
\title{Dynamics of complex-valued modified KdV solitons with applications to the stability of breathers}
\author{Miguel A. Alejo}
\author{Claudio Mu\~noz}
\address{IMPA(Instituto Nacional de Matem\'atica Pura e Aplicada), Rio de Janeiro\\ Brasil}
\email{malejo@impa.br}
\address{CNRS and Laboratoire de Math\'ematiques d'{}Orsay UMR 8628, B\^at. 425 Facult\'e des Sciences d'{}Orsay,
Universit\'e Paris-Sud F-91405 Orsay Cedex France}
\email{claudio.munoz@math.u-psud.fr}
\date{\today}
\thanks{}

\chardef\bslash=`\\ 

\newtheorem{thm}{Theorem}[section]
\newtheorem{cor}[thm]{Corollary}
\newtheorem{lem}[thm]{Lemma}
\newtheorem{prop}[thm]{Proposition}

\newtheorem{defn}[thm]{Definition}

\theoremstyle{remark}
\newtheorem{rem}{Remark}[section]

\numberwithin{equation}{section}

\newcommand{\R}{\mathbb{R}}

\newcommand{\Z}{\mathbb{Z}}

\newcommand{\Com}{\mathbb{C}}

\newcommand{\al}{\alpha}
\newcommand{\bt}{\beta}
\newcommand{\ga}{\gamma}

\newcommand{\re}{\operatorname{Re}}
\newcommand{\ima}{\operatorname{Im}}

\def\bm{\left( \begin{array}{cc}}
\def\endm{\end{array}\right)}

 \providecommand{\abs}[1]{\lvert#1 \rvert}
 \providecommand{\norm}[1]{\lVert#1 \rVert}

\newcommand{\ve}{\varepsilon}

\newcommand{\be}{\begin{equation}}
\newcommand{\ee}{\end{equation}}
\newcommand{\ba}{\left(\begin{array}{c}}
\newcommand{\ea}{\end{array}\right)}
\newcommand{\bea}{\begin{eqnarray}}
\newcommand{\eea}{\end{eqnarray}}
\newcommand{\bee}{\begin{eqnarray*}}
\newcommand{\eee}{\end{eqnarray*}}
\newcommand{\ben}{\begin{enumerate}}
\newcommand{\een}{\end{enumerate}}
\newcommand{\nonu}{\nonumber}

\newcommand{\eval}[2][\right]{\relax
  \ifx#1\right\relax \left.\fi#2#1\rvert}

\let\abs=\envert

\let\norm=\enVert

\begin{document}
\begin{abstract}
We study the \emph{long-time} dynamics of complex-valued modified Korteweg-de Vries (mKdV) solitons, which are recognized because they blow-up in finite time. We establish stability properties at the $H^1$ level of regularity, uniformly away from each blow-up point. These new properties are used to prove that mKdV breathers are $H^1$ stable, improving our previous result \cite{AM}, where we only proved $H^2$ stability. The main new ingredient of the proof is the use of a B\"acklund transformation which relates the behavior of breathers, complex-valued solitons and small real-valued solutions of the mKdV equation. We also prove that negative energy breathers are asymptotically stable. Since we do not use any method relying on the Inverse Scattering Transform, our proof works even under $L^2(\R)$ perturbations, provided a corresponding local well-posedness theory is available. 
\end{abstract}
\maketitle \markboth{Stability of breathers} {Miguel A. Alejo and Claudio Mu\~noz}
\renewcommand{\sectionmark}[1]{}
\tableofcontents

\section{Introduction}

\medskip

Consider the modified Korteweg-de Vries (mKdV) equation on the real line
\be\label{mKdV}
u_t+(u_{xx} + u^3)_x =0,
\ee
where $u=u(t,x)$ is a complex-valued function, and $(t,x)\in \R^2$. Note that \eqref{mKdV} is not $U(1)$-invariant.  In the case of real-valued initial data, the associated Cauchy problem for \eqref{mKdV} is globally well-posed for initial data in $H^s(\R)$, for any $s> \frac14$, see Kenig-Ponce-Vega \cite{KPV}, and Colliander, Keel, Staffilani, Takaoka and Tao \cite{CKSTT}. Additionally, the (real-valued) flow map is not uniformly continuous if $s<\frac 14$ \cite{KPV2}.\footnote{However, one can construct a solution in $L^2$, see \cite{CHT}.} In order to prove this last result, Kenig, Ponce and Vega considered a very particular class of solutions of \eqref{mKdV} called \emph{breathers}, discovered by Wadati in \cite{W1}.

\begin{defn}[See e.g. \cite{W1,La}]\label{mKdVB} 
Let $\al, \bt  >0$ and $x_1,x_2\in \R$ be fixed parameters. The mKdV breather is a smooth solution of \eqref{mKdV} given explicitly  by the formula
\be\label{breather}
\begin{split}
 B:= B(t,x; \al,\bt, x_1,x_2)   := & 2\sqrt{2} \partial_x \Big[ \arctan \Big( \frac{\bt}{\al}\frac{\sin(\al y_1)}{\cosh(\bt y_2)}\Big) \Big] ,
\end{split}
\ee
where
\be\label{y1y2}
y_1 := x+ \delta t + x_1, \quad y_2 : = x+ \ga t + x_2,
\ee
and
\be\label{deltagamma}
\delta := \al^2 -3\bt^2, \quad  \ga := 3\al^2 -\bt^2.
\ee
\end{defn}

\medskip

Breathers are \emph{oscillatory bound states}. They are periodic in time (after a suitable space shift) and localized in space. The parameters $\al$ and $\bt$ are scaling parameters, $x_1,x_2$ are shifts, and $-\ga$ represents the \emph{velocity} of a breather.  As we will see later, the main difference between \emph{solitons}\footnote{See \eqref{Sol}.} and breathers is given at the level of the \emph{oscillatory} scaling $\al$, which is not present in the case of solitons. For a detailed account of the physics of breathers see e.g. \cite{La,AC,Au,Ale,AM} and references therein.

\medskip

Numerical computations (see Gorria-Alejo-Vega \cite{AGV}) showed that breathers are \emph{numerically} stable. Next, in \cite{AM} we constructed a Lyapunov functional that controls the dynamics of $H^2$-perturbations of \eqref{breather}. The purpose of this paper is to improve our previous result \cite{AM} and show that mKdV breathers are indeed $H^{1}$ stable, i.e. stable in the energy space.

\begin{thm}\label{T1} Let $\al, \bt >0$ be fixed scalings. There exist parameters $\eta_0, A_0$, depending on $\al$ and $\bt$ only, such that the following holds.  Consider $u_0 \in H^1(\R)$, and assume that there exists $\eta \in (0,\eta_0)$ such that 
\be\label{In}
\|  u_0 - B(0,\cdot;\al,\bt,0,0) \|_{H^1(\R)} \leq \eta.
\ee
Then there exist functions $x_1(t), x_2(t) \in \R$ such that the solution $u(t)$ of the Cauchy problem for the mKdV equation \eqref{mKdV}, with initial data $u_0$, satisfies
\be\label{Fn1}
\sup_{t \, \in \, \R}\big\| u(t) - B(t,\cdot ; \al,\bt, x_1(t),x_2(t)) \big\|_{H^1(\R)}\leq A_0 \eta,
\ee
\be\label{Fn2}
\sup_{t \, \in \, \R} |x_1'(t)|+|x_2'(t)| \leq CA_0 \eta,
\ee
for some constant $C>0.$
\end{thm}

\medskip

The initial condition (\ref{In}) can be replaced by any initial breather profile of the form $ B(t_0; \al,\bt, x_1^0, x_2^0)$, with $t_0, x_1^0, x_2^0 \in \R$, thanks to the invariance of the equation under translations in time and space.\footnote{Indeed, if $u(t,x)$ solves \eqref{mKdV}, then for any $t_0, x_0\in \R$, and $c>0$, $u(t-t_0, x-x_0)$, $c^{1/2} u(c^{3/2} t, c^{1/2} x)$, $u(-t,-x)$ and $-u(t,x)$ are solutions of (\ref{mKdV}).}  Moreover, using the Miura transform \cite{Kruskal}, one can prove a natural stability property in $L^2(\R;\Com)$ for the associated complex-valued KdV breather.  

\medskip

Additionally, from the proof the shifts $x_1(t)$ and $x_2(t)$ in Theorem \ref{T1} can be described almost explicitly\footnote{See equation \eqref{x1x2}.}, which is a sustainable improvement with respect to our previous result \cite{AM}, where no exact control on the shift parameters was given. We recall that we obtain such a control with no additional decay assumptions on the initial data other than being in $H^1(\R)$.

\medskip

Theorem \ref{T1} places breathers as stable objects at the same level of regularity as mKdV solitons, even if they are very different in nature. To be more precise, a (real-valued) soliton is a solution of \eqref{mKdV} of the form
\be\label{Sol}
u(t,x) = Q_c (x-ct), \quad Q_c(s) := \sqrt{c} \,Q(\sqrt{c} s), \quad c>0,
\ee
with
\[
Q (s):= \frac{\sqrt{2}}{\cosh (s)} =2\sqrt{2} \, \partial_s[\arctan(e^{s})],
\]
and where $Q_c>0$ satisfies the nonlinear ODE
\be\label{ecQc}
Q_c'' -c\, Q_c +Q_c^3=0, \quad Q_c\in H^1(\R).
\ee
We recall that solitons are $H^1$-stable (Benjamin \cite{Benj}, Bona-Souganidis-Strauss \cite{BSS}). See also the works by Grillakis-Shatah-Strauss \cite{GSS} and Weinstein \cite{We2} for the nonlinear Schr\"odinger case. 

\medskip

Even more surprising is the fact that Theorem \ref{T1} will arise as a consequence of a suitable stability property of the zero solution and of {\bf  complex-valued} mKdV solitons, which are singular solutions. 

\medskip

A complex-valued soliton is a solution of the form \eqref{Sol} of \eqref{mKdV}, with a complex-valued  scaling and velocity, i.e., 
\be\label{Qc}
u(t,x): = Q_c(x-ct), \quad \sqrt{c}:= \bt + i\al, \quad \al,\bt>0,
\ee
see Definition \ref{Qab} for a rigorous interpretation. In Lemma \ref{Good} we give a detailed description of the singular nature of \eqref{Qc}. On the other hand, very little is known about mKdV \eqref{mKdV} when the initial data is complex-valued. For instance, it is known that  it has finite time blow-up solutions, the most important examples being the complex solitons themselves, see e.g. Bona-Vento-Weissler \cite{Bona} and references therein for more details. According to \cite{Bona}, blow-up in the complex-valued case can be understood as the intersection with the real line $x\in \R$ of a curve of poles of the solution after being extended to the complex plane (i.e., now $x$ is replaced by $z\in \Com$). Blow-up in this case seems to have better properties than the corresponding critical blow-up described by Martel and Merle in \cite{MaMe}.

\medskip

Let $H^1(\R;\Com)$ denote the standard Sobolev space of complex-valued functions $f(x) \in  \Com$, $x\in \R$. In this paper we prove the following \emph{stability property} for solitons, far away from each blow-up time.

\begin{thm}\label{T2}
There exists an open set of initial data in $H^1(\R;\Com)$ for which the mKdV complex solitons are well-defined and stable in $H^1(\R;\Com)$ for all times uniformly separated from a countable sequence of finite blow-up times with no limit points. Moreover, one can define a mass and an energy, both invariant {\bf for all time}.
\end{thm}

We cannot prove an all-time stability result using the $H^1(\R;\Com)$-norm because even complex solitons leave that space at each blow-up time, and several computations in this paper break down. However, the previous result states that the Cauchy problem is almost globally well-posed around a soliton, and the solution can be continued after  (or before) every blow-up time. The novelty with respect to the local Cauchy theory \cite{KPV} is that now it is possible to define an almost global solution instead of defining a local solution on each subinterval of time defined by two blow-up points, because from the proof we will recognize that the behavior before and after the blow-up time are deeply linked. From this property the existence and invariance of uniquely well-defined mass and energy will be quite natural. For this particular problem, we answer positively the questions about existence, uniqueness  and regularity after blow-up posed by Merle in \cite{Merle}. See Theorem \ref{T2a} and its corollaries for a 
more detailed statement. 

\medskip

We finally prove that breathers behaving as standard solitons are asymptotically stable in the energy space. For previous results for the soliton and multi-soliton case, see Pego-Weinstein \cite{PW} and Martel-Merle \cite{MMnon}.

\begin{thm}\label{T3}
Under the hypotheses of Theorem \ref{T1}, there exists $c_0>0$ depending on $\eta$, with $c_0(\eta) \to 0$ as $\eta\to 0$, such that the following holds. There exist $\bt^*$ and $\al^*$ close enough to $\bt$ and $\al$ respectively (depending on $\eta$), for which  
\be\label{AS1}
\lim_{t\to +\infty} \|u(t) - B(t;\cdot, \al^*,\bt^*, x_1(t),x_2(t))\|_{H^1(x\geq c_0 t)} =0. 
\ee
In particular, the asymptotic of the  solution  $u(t)$ has new and explicit velocity parameters $\delta^* = (\al^*)^2 -3(\bt^*)^2$ and $\ga^* = 3(\al^*)^2 -(\bt^*)^2$ at the main order.
\end{thm}

The previous result is more interesting when $\ga<0$, see \eqref{deltagamma}. In this case, the breather has negative energy (see \cite[p. 9]{AM}), and it moves rightwards in space (the so-called physically relevant region). Theorem \ref{T3} states that breathers almost clean the right portion of the real line. We recall that  working in the energy space implies that small solitons moving to the right in a very slow fashion are allowed (the condition $c_0>0$ is essential, see e.g. Martel-Merle \cite{MMnon}). Indeed, there are explicit solutions of \eqref{mKdV} composed by one breather and one very small soliton moving rightwards, that contradicts any sort of \emph{global} asymptotic stability result in the energy space \cite{La}. Additionally, we cannot ensure that the left portion  of the real line $\{x<0\}$ corresponds to radiation only. Following \cite{La}, it is possible to construct a solution to \eqref{mKdV} composed  by two breathers, one very small with respect to the other one, the latter 
with positive velocity, and the former with small but still negative velocity (just take the corresponding scaling parameters $\al$ and $\bt$ both small such that $-\ga<0$). Such a solution has no radiation at infinity.  Of course, working in a neighborhood of the breather using weighted spaces rules out such small perturbations.

\medskip

The mechanism under which $\al^*$ and $\bt^*$ are chosen is very natural and reflects the power and simplicity of the arguments of the proof: under different scaling parameters, it was impossible to describe the dynamics as in Theorem \ref{T1}. We are indeed under two linked results: in some sense Theorem \ref{T1} is a consequence of Theorem \ref{T3} and vice versa.

\medskip

On the other hand, the fact that no shifts in \eqref{AS1} are needed can be contrasted with the Martel-Merle computations in \cite{MMcol1}. In that paper they calculated the leading order of the shift perturbation for the small-large soliton collision in the mKdV case. It was found that such shifts are very small ($\sim \eta^2$) compared with the size of the perturbation ($\sim \eta$).

\medskip

Finally, concerning the portion of the mass not considered in \eqref{AS1}, we have the following characterization of inelasticity.

\begin{cor}\label{Ine}

Assume that $u_0$ in \eqref{In} is non trivial, i.e.,
\[
\ell_0:= \|  u_0 - B(0,\cdot;\al,\bt,0,0) \|_{H^1(\R)} >0.
\]
Then there exists $c_0>0$ independent of $\eta$ such that 
\[
\liminf_{t\to +\infty} \|u(t) - B(t;\cdot, \al^*,\bt^*, 0,0)\|_{H^1(\R)}  \geq   c_0 \ell_0.
\]
Moreover, we have 
\[
c_0 \ell_0 \leq |\bt^* -\bt| +|\al^* -\al| \leq \frac{1}{c_0} \ell_0.
\]
\end{cor}

\medskip

It is also important to emphasize that \eqref{mKdV} is a well-known \emph{completely integrable} model \cite{Kruskal,AC,La,LAX1,Sch}, with infinitely many conserved quantities, and a suitable Lax-pair formulation. The Inverse Scattering Theory has been applied in \cite{Sch} to describe the evolution of \emph{rapidly decaying} initial data, by purely algebraic methods. Solutions are shown to decompose into a very particular set of solutions: solitons, breathers and radiation. Moreover, as a consequence of the integrability property, these nonlinear modes interact elastically during the dynamics, and no dispersive effects are present at infinity. In particular, even more complex solutions are present, such as \emph{multi-solitons} (explicit solutions describing the interaction of several solitons \cite{HIROTA1}).  Multi-solitons for mKdV and several integrable models of Korteweg-de Vries type are stable in $H^1$, see Maddocks-Sachs \cite{MS} for the KdV case and in a more general setting see Martel-Merle-Tsai 
\cite{MMT}. 

\medskip

However, the proof of Theorem \ref{T1} does not involve any method relying on the Inverse Scattering transform \cite{Kruskal, Sch}, nor the steepest descent machinery \cite{Deift},\footnote{Note that in \cite{Deift} the authors consider the \emph{defocusing} mKdV equation, which has no smooth solitons and breathers.} which allows to work in the \emph{very large} energy space $H^1(\R)$. Note that  if the Inverse Scattering methods are allowed, one could describe the dynamics of very general initial data with more detail. But if this is the case, additional decay and/or spectral assumptions are always needed, and except by well-prepared initial data, such conditions are difficult to verify. We claim that our proof works even if the initial data is in $L^2(\R)$, provided mKdV is locally well-posed at that level of regularity, which remains a very difficult open problem. 

\medskip

Comparing with \cite{AM}, where we have proved that mKdV breathers are $H^2$-stable, now we are not allowed to use the \emph{third} conservation law associated to mKdV:\footnote{See \eqref{M1} and \eqref{E1} for the other two low-regularity conserved quantities.}
\[
F[u](t) = \frac 12 \int_\R u_{xx}^2(t,x)dx -\frac 52 \int_\R u^2 u_x^2(t,x)dx +\frac 14\int_\R u^6(t,x)dx,
\]
nor the elliptic equation satisfied by {\bf any} breather profile:
\[
B_{(4x)} -2(\bt^2 -\al^2) (B_{xx} + B^3)  +(\al^2 +\bt^2)^2 B + 5 BB_x^2 + 5B^2 B_{xx} + \frac 32 B^5 =0,
\]
since the dynamics is no longer in $H^2$. Moreover, since breathers are bound states, there is no associated decoupling in the dynamics as time evolves as in the Martel-Merle-Tsai paper \cite{MMT}, which makes the proof of the $H^1$ case even more difficult. We need a different method of proof. 

\medskip

In this paper we follow a method of proof that it is in the spirit of the seminal work by Merle and Vega \cite{MV} (see also Alejo-Mu\~noz-Vega \cite{AMV}), where the $L^2$-stability of KdV solitons has been proved. In those cases the use of the Miura and Gardner transformations were the new ingredients to prove stability where the standard energy is missing. Recently, the Miura transformation has been studied at very low regularity. Using this information, Buckmaster and Koch showed that KdV solitons are stable even under $H^{-1}\cap H^{-3/4}$ perturbations \cite{BK}. 

\medskip

More precisely, in this paper we will make use of the B\"acklund transformation \cite[p. 257]{La} associated to mKdV to obtain new conserved quantities, additional to the mass and energy. We point out the recent works by Mizumachi-Pelinovsky \cite{MP} and Hoffmann-Wayne \cite{HW}, where a similar approach was described for the NLS and sine-Gordon equations and their corresponding \emph{one-solitons}. However, unlike those previous works, and in order to control any breather, we use the B\"acklund transformation twice: one to control an associated complex-valued mKdV soliton, and a second one to get  almost complete control of the breather.

\medskip

Indeed, solving the B\"acklund transformation in a vicinity of a breather leads (formally) to the emergence of complex-valued mKdV solitons, which blow-up in finite time. A difficult problem arises at the level of the Cauchy theory, and any attempt to prove stability must face the ill-posedness behavior of the complex-valued mKdV equation \eqref{mKdV}. However, after a new use of the B\"acklund transformation around the complex soliton we end-up with a small, real-valued $H^1(\R)$ solution of mKdV, which is stable for all time. The fact that a second application of the B\"acklund transformation leads to a real-valued solution is not trivial and is a consequence of a deep property called \emph{permutability theorem} \cite{La}. Roughly speaking, that result states that the order under which we perform two inversions of the B\"acklund transformation does not matter.  After some work we are able to give a rigorous proof of the following fact: we can invert a breather using B\"acklund towards two 
particularly well chosen 
complex solitons first, and then invert once again to obtain two small solutions --say $a$ and $b$--, and the final result \emph{must be the same}. Even better, one can show that $a$ has to be the conjugate of $b$, which gives the real character of the solution. Now the dynamics is real-valued and simple. We use the Kenig-Ponce-Vega \cite{KPV} theory to evolve the system to any given time. Using this trick we avoid dealing with the blow-up times of the complex soliton --for a while-- and at the same time we prove a new stability result for them.

\medskip

However, unlike the previous results \cite{MP,HW}, we cannot invert the B\"acklund transformation at any given time, and in fact each blow-up time of the complex-valued mKdV soliton is a dangerous obstacle for the breather stability. In order to extend the stability property up to the blow-up times we discard the method involving  the B\"acklund transformation. Instead we run a bootstrap argument starting from a fixed time very close to each singular point, using the fact the real-valued mKdV dynamics is continuous in time. Finally, using energy methods related to the stability of single solitons we are able to extend the uniform bounds in time to any singularity point, with a universal constant $A_0$ as in Theorem \ref{T1}. 

\medskip

From the proof it will be evident that even if there is no global well-posedness theory (with uniform bounds in time) below $H^s$, $s<\frac 14$, one can prove stability of breathers in spaces of the form $H^1 \cap H^s$, $s<\frac 14$, following the ideas of Buckmaster and Koch \cite{BK}. We thank professor Herbert Koch for mentioning to us this interesting property. 

\medskip

Our results apply without important modifications to the case of the sine-Gordon (SG) equation in $\R_t \times \R_x $
\be\label{SG}
u_{tt} -u_{xx} + \sin u =0, \quad (u,u_t)(t,x) \in \R^2,
\ee
and its corresponding breather \cite[p. 149]{La}. See \cite{BMW,D,SW} for related results. Note that SG is globally well-posed in $L^2\times H^{-1}$; then we have that breathers are stable under small perturbations in that space. Since the proofs are very similar, and in order to make this paper non redundant, we skip the details.

\medskip

Moreover, following our proof it is possible to give a new proof of the global $H^1$-stability of two-solitons proved by Martel, Merle and Tsai in \cite{MMT}.

\medskip

We also claim that $k$-breathers ($k\geq 2$), namely solutions composed by $k$ different breathers are also $H^1$-stable. Following the proof of Theorem \ref{T1}, one can show by induction that  a $k$-breather can be obtained from a $(k-1)$-breather after two B\"acklund transformations using a fixed set of complex conjugate parameters, as in Lemmas \ref{Equal0} and \ref{Equal}. After proving this identity, the rest of the proof adapts with no deep modifications.

\medskip

This paper is organized as follows. In Section \ref{2} we introduce the complex-valued soliton profiles. Section \ref{3} is devoted to the study of the mKdV B\"acklund transformation in the vicinity of a given complex-valued mKdV solution. In Section \ref{4} we apply the previous results to prove Theorem \ref{T2} (see Theorem \ref{T2a}). Section \ref{5} deals with the relation between complex soliton profiles and breathers. In Section \ref{6} we apply the results from Section \ref{3} to the case of a perturbation of a breather solution. Finally, in Sections \ref{7} and \ref{8} we prove Theorems \ref{T1} and \ref{T3} and Corollary \ref{Ine}.

\medskip

\noindent
{\bf Acknowledgments.} The first author acknowledges the financial support from the University of the Basque Country-EHU. Part of this work was done while  the second author was L.E. Dickson Instructor at the University of Chicago. He  would like to express his gratitude with the members of the Department of Mathematics, in particular professors Carlos Kenig and Wilhelm Schlag. We also thank Herbert Koch and Yvan Martel for several enlightening  discussions which improved the quality of this paper, and a gap in the first version of this one. 

\bigskip

\section{Complex-valued mKdV soliton profiles}\label{2}
 
\medskip

First of all, we recall the well-known complex-valued mKdV profile.

\begin{defn}\label{Qab}
Consider parameters $\al,\bt>0$, $ x_1$ and $ x_2 \in \R$. We introduce the complex-valued kink profile
\[
\widetilde Q =\widetilde Q(x; \al,\bt, x_1, x_2), 
\]
defined as
\be\label{tQ}
\widetilde Q  :=  2\sqrt{2}  \arctan \big( e^{ \bt y_2 + i \al y_1 }\big),
\ee
where $y_1$ and $y_2$ are (re)defined as
\be\label{y1y2mod}
y_1 : = x+ x_1, \quad y_2 := x + x_2.
\ee
Note that
\be\label{tQinfty}
\widetilde Q(-\infty) = 0, \quad \widetilde Q(+\infty) = \sqrt{2}\pi. 
\ee   
We define the complex-valued soliton profile as follows:
\bea
Q &:= & \partial_x \widetilde Q \nonu \\
& =& \frac{2\sqrt{2}  (\bt+i\al) e^{ \bt y_2 + i \al y_1 } }{ 1+e^{ 2(\bt y_2 + i \al y_1 )}} \label{Q0}\\
& =& \sqrt{2}\frac{\bt  \cosh(\bt y_2) \cos(\al y_1) + \al  \sinh (\bt y_2) \sin(\al y_1)}{\cosh^2(\bt y_2)  -  \sin^2(\al y_1)} +  \nonu \\
& &  + i  \sqrt{2} \frac{\al  \cosh(\bt y_2) \cos(\al y_1) - \bt  \sinh (\bt y_2) \sin(\al y_1)}{\cosh^2(\bt y_2)-  \sin^2(\al y_1)}. \label{Q} 
\eea
Finally we denote
\be\label{tQt0} 
\widetilde Q_t :=   -(\bt+i\al)^2 Q,
\ee
and
\be\label{tQ12}
\widetilde Q_1 := \partial_{x_1} \widetilde Q, \quad \widetilde Q_2 := \partial_{x_2} \widetilde Q.
\ee
\end{defn}

Note that $Q$ is complex-valued and is pointwise convergent to the soliton $Q_{\bt^2}$ as $\al \to 0$.  A second condition satisfied by $\widetilde Q$ and $Q$ is the following periodicity property: for all $k\in \Z$,
\be\label{perio}
\begin{cases}
\displaystyle{\widetilde Q(x; \al,\bt, x_1 + \frac{k \pi}\al, x_2) =(-1)^k \widetilde Q(x; \al,\bt, x_1, x_2), }\\
\hbox{ \;}\\
\displaystyle{Q(x; \al,\bt, x_1 + \frac{k \pi}\al, x_2) =(-1)^k  Q(x; \al,\bt, x_1 , x_2).}
\end{cases}
\ee

\medskip

In what follows, we remark that $\widetilde Q$ and $Q$ may blow-up in finite time. 

\begin{lem}\label{Good}
Consider the complex-valued soliton profile defined in \eqref{tQ}-\eqref{Q}. Assume that 
\be\label{BadCondFF}
\hbox{for $ x_2$ fixed and some $k\in \Z$}, \quad x_{1} =  x_2  +\frac{\pi}{\al} \Big(k+\frac 12 \Big).
\ee
Then $\widetilde Q$ and $Q$ cannot be defined at $x= - x_2.$ Moreover, if $x_1 =x_2 =0$, we have $Q(\cdot; \al,\bt,0,0)\in H^1(\R;\Com)$.
\end{lem}

\begin{rem}
We emphasize that, given $x_2$ fixed, the set of points $x_1$ of the form \eqref{BadCondFF} for some $k\in \Z$ is a countable set of real numbers with no limit points.
\end{rem}

\begin{proof}
Fix $ x_2 \in \R$. If \eqref{BadCondFF} is satisfied for some $k\in \Z$,  we have that at $x=- x_2,$
\[
y_1= x + x_1 = \frac{\pi}{\al} \Big(k+\frac 12\Big), \quad y_2 = x+ x_2=0,
\]
and
\be\label{shco}
\sinh(\bt y_2) =0, \quad \cos(\al y_1) =0. 
\ee
Therefore, under \eqref{BadCondFF}, we have from \eqref{tQ} and \eqref{Q} that  $\widetilde Q$ and $Q$ cannot be defined at $x=- x_2$. Finally, if $x_1=x_2=0$, we have
\[
k+\frac 12 =0, \quad k\in \Z,
\]
which is impossible.
\end{proof}

\begin{lem}
Fix  $\al,\bt>0$ and $x_1,x_2\in \R$ such that \eqref{BadCondFF} is not satisfied. Then  we have
\be\label{ecQ}
Q_{xx} - (\bt+i\al)^2 Q +Q^3 =0, \quad \hbox{for all } x\in \R,
\ee
and
\be\label{Qx2}
Q_x^2 -(\bt+i\al)^2 Q^2 + \frac 12Q^4 =0, \quad \hbox{for all } x\in \R.
\ee
Moreover, the previous identities can be extended to any $x_1,x_2\in \R$ by continuity.
\end{lem}

\begin{proof}
Direct from the definition.
\end{proof}

Assume that \eqref{BadCondFF} does not hold. Consider the $\sin$ and $\cos$ functions applied to complex numbers. We have from \eqref{tQ} and \eqref{Q0},
\bee
 \sin \Big(\frac{\widetilde Q}{\sqrt{2}}\Big) &=& \sin (2 \arctan e^{\bt y_2 + i\al y_1}) \\
 &=&2e^{\bt y_2 + i\al y_1} \cos^2(\arctan e^{\bt y_2 + i\al y_1})\\
 & =& \frac{ 2e^{\bt y_2 + i\al y_1} }{1+ e^{2(\bt y_2 + i\al y_1)}}=  \frac 1{\bt+i\al}\frac Q{\sqrt{2}}.
\eee
Similarly, from this identity we have
\[
Q_x - (\bt + i\al) \cos \Big( \frac{\widetilde Q}{\sqrt{2}}\Big) Q=0,
\]
so that from \eqref{tQt0} and \eqref{Qx2},
\bee
&  & \widetilde Q_t + (\bt+ i\al)  \Big[  Q_x  \cos \Big(\frac{\widetilde Q}{\sqrt{2}} \Big)   + \frac { Q^2}{\sqrt{2}}\sin \Big( \frac{\widetilde Q}{\sqrt{2}} \Big) \Big]\\
& & \qquad = -(\bt + i\al)^2Q  +    Q_x^2 Q^{-1}   +  \frac 12Q^3  = 0.
\eee
So far, we have proved the following result.

\begin{lem}\label{Equal0}
Let $Q$ be a complex-valued soliton profile with scaling parameters $\al,\bt > 0$ and shifts $x_1,x_2\in \R$, such that \eqref{BadCondFF} is not satisfied. Then we have
\be\label{zero0}
\frac Q{\sqrt{2}}  - (\bt + i\al) \sin \Big(\frac{\widetilde Q}{\sqrt{2}}\Big) \equiv 0 ,
\ee
and
\be\label{zero1}
 \widetilde Q_t    + (\bt+ i\al)  \Big[  Q_x  \cos \Big(\frac{\widetilde Q}{\sqrt{2}} \Big)   + \frac { Q^2}{\sqrt{2}}\sin \Big( \frac{\widetilde Q}{\sqrt{2}} \Big) \Big]  \equiv 0,
\ee
where $\sin z$ and $\cos z$  are defined on the complex plane in the usual sense. 
\end{lem}

We finish with a simple computational lemma.

\begin{lem}
Fix $x_1,x_2$ such that \eqref{BadCondFF} is not satisfied. Then, for all $\al,\bt \neq 0$  we have
\be\label{N1}
\mathcal N := \frac12\int_{-\infty}^x Q^2 =\frac{ 2(\bt+i\al) \, e^{ 2(\bt y_2 + i\al y_1 )} }{1+ e^{2(\bt y_2 + i\al y_1)} }, 
\ee
and  
\be\label{MassQ}
\frac 12  \int_\R Q^2 = 2(\bt+i\al),
\ee
no matter what are  $x_1, x_2$.
Finally, if $L_1:= \log(1+ e^{2(\bt x_2 + i\al x_1)})$,
\be\label{LogL}
\int_{0}^x \mathcal N =  \log (1+ e^{2(\bt y_2 + i\al y_1) }) - L_1. 
\ee
Note that the previous formula is well-defined since $x_1$ and $x_2$ do not satisfy \eqref{BadCondFF}.
\end{lem}

\begin{proof}
It is not difficult to check that \eqref{N1} is satisfied. Note that 
\[
\lim_{x\to -\infty}\abs{\frac{ 2(\bt+i\al) \, e^{ 2(\bt y_2 + i\al y_1 )} }{1+ e^{2(\bt y_2 + i\al y_1)} }} =0.
\]
Identity \eqref{MassQ} is a consequence of the fact that
\[
\lim_{x\to + \infty} \frac{ 2(\bt+i\al) \, e^{ 2(\bt y_2 + i\al y_1 )} }{1+ e^{2(\bt y_2 + i\al y_1)} } = 2(\bt+i\al) .
\]
Finally, \eqref{LogL} is easy to check.
\end{proof}

\bigskip

\section{B\"acklund transformation for mKdV}\label{3}

\medskip

The previous properties (i.e., Lemma \ref{Equal0}) are consequence of a deeper result. In what follows, we fix a primitive $\tilde f$ of $f$, i.e., 
\be\label{tildeN}
\widetilde f_x := f,
\ee
where $f$ is assumed only in $L^2(\R)$. Notice that even if $f =f(t,x)$ is a solution of mKdV, then a corresponding term $\widetilde f(t,x)$ may be unbounded in space. 
\begin{defn}[See e.g. \cite{La}]
Let
\[
(u_a,u_b,v_a,v_b,m )\in H^1(\R;\Com)^2 \times H^{-1}(\R;\Com)^2 \times \Com.
\]
We set
\[
G: =(G_1,G_2), \quad G=G(u_a,u_b,v_a,v_b,m),
\]
where 
\be\label{G1}
 G_1(u_a,u_b,v_a,v_b,m) := \frac{(u_a-u_b)}{\sqrt{2}}  - m \sin \Big(\frac{\widetilde u_a+\widetilde u_b}{\sqrt{2}}\Big),
\ee
and
\bea\label{G2}
& & G_2(u_a,u_b,v_a,v_b,m)   :=   v_a-v_b   \nonu\\
& &  \qquad + \, m \,  \Big[ ( (u_a)_{x} +  (u_b)_{x}) \cos \Big(\frac{\widetilde u_a+\widetilde u_b}{\sqrt{2}} \Big)   + \frac { (u_a^2 + u_b^2)}{\sqrt{2}}\sin \Big( \frac{\widetilde u_a+\widetilde u_b}{\sqrt{2}} \Big) \Big] .
\eea
\end{defn}

For the moment we do not specify the  space where $G(u_a,u_b,  v_a, v_b,m)$ takes place if  $(u_a,u_b,v_a,v_b,m )\in H^1(\R;\Com)^2 \times H^{-1}(\R;\Com)^2 \times \Com$. However, thanks to Lemma \ref{Equal0}, we have the following result.

\begin{lem}
Assume that $x_1$ and $x_2$ do not satisfy \eqref{BadCondFF}. We have
\[
G( Q, 0,\widetilde Q_t, 0, \bt+i\al) \equiv (0,0).
\]
\end{lem}
Note that the previous identity can be extended by zero to the case where $x_1$ and $x_2$ satisfy \eqref{BadCondFF}, in such a form that now $G( Q, 0,\widetilde Q_t, 0, \bt+i\al)$, as a function of $(x_1,x_2) \in \R^2$, is now well-defined and continuous everywhere.

\medskip

In what follows we consider the invertibility of the B\"acklund transformation on complex-valued functions. See \cite{HW} for the statement involving the real-valued solitons in the sine-Gordon case and \cite{MP} for the case of nonlinear Schr\"odinger solitons.

\begin{prop}\label{IFT}
 Let $X^0:= (u_a^0,u_b^0, v_a^0,v_b^0,m^0) \in H^1(\R; \Com)^2 \times H^{-1}(\R;\Com)^2 \times \Com$ be such that
 \be\label{rem0}
 \re m^0 >0,
 \ee
 
\be\label{Conds11}
G(X^0)=(0,0),    
\ee

\be\label{Conds12}
 \sin\Big( \frac{\widetilde u_a^0 + \widetilde u_b^0}{\sqrt{2}}\Big) \in H^1(\R;\Com),
\ee
and

\be\label{Conds13}
\lim_{-\infty} \ ( \widetilde u_a^0 + \widetilde u_b^0) =0, \quad   \lim_{+\infty}\ (\widetilde u_a^0 + \widetilde u_b^0) =\sqrt{2} \pi.
\ee
Assume additionally that the ODE
\be\label{Mu0}
\mu^0_x-  m^0 \cos \Big(\frac{\widetilde u_a^0+\widetilde u_b^0}{\sqrt{2}}\Big) \mu^0 =0 ,
\ee
has a {\bf smooth} solution  $\mu^0=\mu^0(x)\in \Com$ satisfying 
\be\label{Cond110}
 \mu^0 \in H^1(\R;\Com), \quad |\mu^0(x)|>0, \quad \abs{\frac{\mu^0_x(x)}{\mu^0(x)}} \leq C, 
\ee
and
\be\label{Cond111}
\int_\R \sin \Big(\frac{\widetilde u_a^0+\widetilde u_b^0}{\sqrt{2}}\Big)\mu^0 \neq 0.
\ee

\smallskip

\noindent
Then there exist $\nu_0>0$ and $C>0$ such that the following is satisfied. For any $0<\nu<\nu_0$ and any $(u_a,v_a)\in H^1(\R; \Com) \times H^{-1}(\R;\Com)$ satisfying 

\be\label{ceroCond}
\|u_a-u_a^0\|_{H^1(\R;\Com)}  <\nu,
\ee
\smallskip

\noindent
$G$ is well-defined in a neighborhood of $X^0$ and there exists an unique $(u_b,v_b,m)$ defined in an {\bf open} subset of $ H^1(\R,\Com)\times H^{-1}(\R;\Com) \times \Com$ such that 

\be\label{cero}
G(u_a,u_b,v_a,v_b,m) \equiv (0,0),
\ee

\be\label{uno}
\|\widetilde u_a  +\widetilde u_b - \widetilde u_a^0  - \widetilde u_b^0\|_{H^2(\R;\Com)} \leq C\nu,
\ee

\be\label{dos}
\|u_b -u_b^0\|_{H^1(\R;\Com)}  + |m-m^0|< C\nu,
\ee

\be\label{tres}
 \sin\Big( \frac{\widetilde u_a + \widetilde u_b}{\sqrt{2}}\Big) \in H^1(\R;\Com),
\ee
and
\be\label{lim0}
\lim_{- \infty}  \, (\widetilde u_a+\widetilde u_b ) =0, \quad  \lim_{+ \infty}  \, (\widetilde u_a+\widetilde u_b ) = \sqrt{2}\pi.
\ee

\end{prop}

\begin{proof}

Given $u_a,  u_b$, $m$ and $v_a$ well-defined, $v_b$ is uniquely defined from \eqref{G2}. We solve for $ u_b$ and $m$ now. We will use the Implicit Function Theorem.  

\medskip

We make a change of variables in order to specify a suitable range for $G$ and being able to prove \eqref{lim0}. Define
\be\label{uc}
u_c:= u_a +  u_b - u_c^0,  \qquad {u}_c^0 := u_a^0 + u_b^0 \in H^1(\R;\Com),
\ee
and similar for $\widetilde u_c$ and $\widetilde u_c^0$:
\[
(\widetilde u_c)_x =u_c, \qquad   (\widetilde u_c^0)_x =u_c^0.
\]
In what follows, we will look for a suitable $\widetilde u_c$ with decay, and then we find $u_b$. Indeed, note that given $u_c$ and $u_a$, $u_b$ can be easily obtained.  Then, with a slight abuse of notation, we consider $G$ defined as follows:
\[
G = (G_1,G_2), \quad G=G(u_a, \widetilde u_c, v_a, v_b, m),
\]
and 
\[
 \begin{array}{ccc} G:  H^1(\R;\Com) \times H^2(\R;\Com) \times H^{-1}(\R;\Com)^2 \times \Com &  \longrightarrow &H^1(\R;\Com) \times H^{-1}(\R;\Com) ,\\ 
  (u_a,  \widetilde u_c, v_a, v_b, m)   & \longmapsto & G(u_a,  \widetilde u_c, v_a, v_b, m) \end{array} 
\]
where, from \eqref{G1},
\be\label{G1new}
G_1(u_a, \widetilde u_c,v_a,v_b,m):=   \frac{(2u_a - u_c^0 - u_c)}{\sqrt{2}}  - m \, \sin \Big(\frac{\widetilde u_c^0 + \widetilde u_c}{\sqrt{2}}\Big),
\ee
and from \eqref{G2},
\bea\label{G2new}
& & G_2(u_a, \widetilde u_c,  v_a, v_b,m)   :=   v_a- v_b   \nonu\\
& &  \qquad + \, m \,  \Big[ (u_c^0 + u_c)_{x} \cos \Big(\frac{\widetilde u_c^0 + \widetilde u_c}{\sqrt{2}} \Big)   + \frac { (u_a^2 + (u_c^0 + u_c -u_a)^2)}{\sqrt{2}}\sin \Big( \frac{\widetilde u_c^0 +\widetilde u_c}{\sqrt{2}} \Big) \Big] . \nonu\\
& & 
\eea
Clearly  $G$ as in \eqref{G1new}-\eqref{G2new} defines a $C^1$ functional in a small neighborhood of $X^1$ given by 
\be\label{X1}
X^1:= (u_a^0, 0, v_a^0, v_b^0, m^0) \in H^1(\R;\Com) \times H^2(\R;\Com) \times H^{-1}(\R;\Com)^2 \times \Com,
\ee
where $G$ is well-defined according to \eqref{Conds12}. Let us apply the Implicit Function Theorem at this point. From \eqref{G1new}  we have to show that 
\[
u_c  + m^0 \cos \Big(\frac{\widetilde u_c^0}{\sqrt{2}}\Big) \widetilde u_c  =f - m  \sin \Big(\frac{\widetilde u_c^0}{\sqrt{2}}\Big) 
\] 
has a unique solution $(\widetilde u_c,m)$ such that $ \widetilde u_c \in H^2(\R;\Com)$, for any $f \in H^1(\R;\Com)$ with linear bounds. From \eqref{Conds13} we have
\be\label{3p16}
\lim_{x\to \pm\infty}\cos\Big(\frac{\widetilde u_c^0}{\sqrt{2}}\Big) = \mp 1,
\ee
so that we can assume 
\[
\mu^0 (x)= \exp \Big(m^0\int_{0}^x \cos\Big(\frac{\widetilde u_c^0}{\sqrt{2}}\Big)\Big).
\]
Note that $\mu^0$ decays exponentially in space as $x\to \pm \infty$. We have
\[
\mu^0 u_c + (\mu^0)_x\widetilde u_c = \mu^0  \Big[ f- m \sin \Big(\frac{\widetilde u_c^0}{\sqrt{2}}\Big) \Big].
\]
Using \eqref{Cond111}, we choose $m \in \Com$ such that 
\be\label{nullity}
\int_\R \mu^0  \Big[ f- m \sin \Big(\frac{\widetilde u_c^0}{\sqrt{2}}\Big) \Big] =0,
\ee
so that
\[
|m| \leq C\|f\|_{L^2(\R;\Com)},
\]
with $C>0$ depending on the quantity $\displaystyle{\abs{\int_\R \mu^0   \sin \Big(\frac{\widetilde u_c^0}{\sqrt{2}}\Big)} } \neq 0$ and $\|\mu^0\|_{L^2(\R;\Com)}$.\footnote{Note that $\|\mu^0\|_{L^2(\R;\Com)}$ blows up as \eqref{BadCondFF} is attained.} We get
\[
 \widetilde u_c = \frac1{\mu^0} \int_{-\infty}^x \mu^0 \Big[ f- m \sin \Big(\frac{\widetilde u_c^0}{\sqrt{2}}\Big) \Big].
\]
Finally, note that we have $u_c\in H^1(\R;\Com)$. Indeed, first of all, thanks to \eqref{nullity}, \eqref{Mu0} and \eqref{3p16},
\[
\lim_{x\to \pm\infty}  \widetilde u_c = \lim_{x\to \pm\infty}\frac{\mu^0}{ \mu^0_x} \Big[ f- m \sin \Big(\frac{\widetilde u_c^0}{\sqrt{2}}\Big) \Big]=0.
\]
Note that if $s\leq x\ll -1$, from \eqref{3p16} we get
\[
\abs{\frac{\mu^0(s)}{\mu^0(x)}}= \abs{\exp \Big(-m^0\int_{s}^x \cos\Big(\frac{\widetilde u_c^0}{\sqrt{2}}\Big)\Big)} \leq C e^{-\re m^0 (x-s)},
\]
so that we have for $x<0$ and large\footnote{Here the symbol $\star$ denotes \emph{convolution}.}
\bee
|\widetilde u_c(x)| & \leq & C  \int_{-\infty}^x e^{-(\re m^0) (x - s)} \Big| f (s)- m \sin \Big(\frac{\widetilde u_c^0(s)}{\sqrt{2}}\Big) \Big|ds \\
&  \leq &   C {\bf 1}_{(-\infty,x]}  e^{-(\re m^0)(\cdot)} \star \Big| f - m \sin \Big(\frac{\widetilde u_c^0}{\sqrt{2}}\Big) \Big| ,  \qquad \re m^0>0. 
\eee
A similar result holds for $x>0$ large, after using \eqref{nullity}. Therefore, from the Young's inequality,
\[
\|\widetilde u_c\|_{L^2(\R;\Com)} \leq C \Big\|  f - m \sin \Big(\frac{\widetilde u_c^0}{\sqrt{2}}\Big) \Big\|_{L^2(\R;\Com)} \leq C\|f\|_{L^2(\R;\Com)},
\]
as desired. On the other hand,
\[
 (\widetilde u_c)_x = \Big[ f- m \sin \Big(\frac{\widetilde u_c^0}{\sqrt{2}}\Big) \Big] -  \frac{\mu^0_x}{(\mu^0)^2} \int_{-\infty}^x \mu^0 \Big[ f- m \sin \Big(\frac{\widetilde u_c^0}{\sqrt{2}}\Big) \Big].
\]
Since $\mu^0_x/\mu^0$ is bounded (see \eqref{Cond110}), we have $\widetilde u_c \in H^1(\R;\Com)$. Finally, it is easy to see that $ \widetilde u_c \in H^2(\R;\Com).$  Note that the constant involving the boundedness of the linear operator $f\mapsto \widetilde u_c$ depends on the $H^1$-norm of $\mu^0$, which blows up if \eqref{BadCondFF} is satisfied.

\medskip

It turns out that we can apply the Implict Function Theorem to the operator $G$ described in \eqref{G1new}-\eqref{G2new}, such that \eqref{cero} is satisfied, provided \eqref{ceroCond} holds. 

\medskip

First of all, let us prove \eqref{tres} and \eqref{lim0}. Note that $\tilde u_c \in H^2(\R;\Com)$, so that \eqref{tres} and \eqref{lim0} are easily obtained.

\medskip

On the other hand, estimate \eqref{uno} is equivalent to the estimate
\[
\|\tilde u_c \|_{H^2(\R;\Com)} \leq C\nu.
\]
We will obtain this estimate using the almost linear character of the operator $G$ around the point $X^1$. Since $\tilde u_c$ satisfies the equation \eqref{G1new}, we have
\[
\frac{(2(u_a-u_a^0) - (\widetilde u_c)_x)}{\sqrt{2}}  - m \Big[ \sin \Big(\frac{\widetilde u_c^0 + \widetilde u_c}{\sqrt{2}}\Big) - \sin \Big(\frac{\widetilde u_c^0 }{\sqrt{2}}\Big) \Big]=0,
\]
or
\[
\frac{(2(u_a-u_a^0) - (\widetilde u_c)_x)}{\sqrt{2}}  - m \Big[ \sin \Big(\frac{\widetilde u_c^0}{\sqrt{2}}\Big) \Big\{ \cos\Big(\frac{\widetilde u_c}{\sqrt{2}} \Big) -1\Big\}   + \cos\Big(\frac{\widetilde u_c^0 }{\sqrt{2}}\Big)  \sin \Big(\frac{\widetilde u_c}{\sqrt{2}}\Big)  \Big]=0.
\]
From \eqref{ceroCond} we know that $\|u_a-u_a^0\|_{H^1(\R;\Com)}  <\nu$.  Since $\widetilde u_c$ is already small in $H^2(\R;\Com)$, it is small in $L^\infty(\R;\Com)$, therefore we have
\[
\abs{ \sin \Big(\frac{\widetilde u_c}{\sqrt{2}}\Big) -\frac{\widetilde u_c}{\sqrt{2}} } \leq C|\widetilde u_c|^3,
\]
and
\[
\abs{ \cos \Big(\frac{\widetilde u_c}{\sqrt{2}}\Big) -1} \leq C |\widetilde u_c|^2.
\]
In other words,
\[ 
(\widetilde u_c)_x + m  \cos\Big(\frac{\widetilde u_c^0 }{\sqrt{2}}\Big)\widetilde u_c = O_{H^1(\R;\Com)}(\nu) +O( |\widetilde u_c|^2).
\]
If we define
\[
\mu (x) := \exp \Big(m \int_{0}^x \cos\Big(\frac{\widetilde u_c^0}{\sqrt{2}}\Big)\Big),
\]
which is exponentially decreasing in space, the solution to system satisfies
\[
|\widetilde u_c| \leq  \mu^{-1} \int_{0}^x \mu [ O_{H^1(\R;\Com)}(\nu) +O( |\widetilde u_c|^2) ]
\]

The proof is complete.
\end{proof}

Later we will need a second invertibility theorem. This time we assume that $m$ is fixed, $u_b \sim u_b^0$ is known and we look for $u_a \sim u_a^0$ this time. Note that the positive sign in front of \eqref{G1} will be essential for the proof, otherwise we cannot take $m$ fixed.

\begin{prop}\label{IFT2}
 Let $X^0= (u_a^0,u_b^0, v_a^0,v_b^0,m^0) \in H^1(\R; \Com)^2 \times H^{-1}(\R,\Com) \times \Com$ be such that \eqref{rem0}, \eqref{Conds11}, \eqref{Conds12} and \eqref{Conds13} are satisfied.
%
%
Assume additionally that the ODE
\be\label{Mu2}
(\mu^1)_x +   m \cos \Big(\frac{\widetilde u_a^0+\widetilde u_b^0}{\sqrt{2}}\Big) \mu^1 =0 ,
\ee
has a {\bf smooth} solution  $\mu^1=\mu^1(x)\in \Com$ satisfying 
\be\label{Cond212}
|\mu^1(x)|>0, \qquad \abs{\frac{\mu^1_x(x)}{\mu^1(x)}}\leq C, \quad \frac{1}{\mu^1} \in H^1(\R;\Com),
\ee
and $G$ is smooth in a small neighborhood of $X^0$. Then there exists $\nu_1>0$ and a fixed constant $C>0$ such that for all $0<\nu<\nu_1$ the following is satisfied. For any $(u_b,v_b,m)\in H^1(\R; \Com) \times H^{-1}(\R;\Com) \times \Com$ such that 
\be\label{Smallness}
\|u_b-u_b^0\|_{H^1(\R;\Com)}  + |m-m^0| <\nu,
\ee 
$G$ is well-defined and there exist unique $(u_a,v_a)\in H^1(\R,\Com)\times H^{-1}(\R;\Com)$ such that 

\[
G(u_a,u_b,v_a,v_b, m) \equiv (0,0),
\]

\be\label{ortho}
\int_\R (u_a-u_b) (\frac1{\mu^1})_x =0.
\ee

\be\label{Conds234}
\|  \widetilde u_a + \widetilde u_b - \widetilde u_a^0 - \widetilde u_b^0 \|_{H^2(\R;\Com)} \leq C\nu,
\ee

\be\label{Conds23a}
\lim_{-\infty} \ ( \widetilde u_a + \widetilde u_b) =0, \quad   \lim_{+\infty} \ (\widetilde u_a + \widetilde u_b) =\sqrt{2} \pi,
\ee
and
\be\label{Cua}
\|u_a-u_a^0\|_{H^1(\R;\Com)}  <C \nu.
\ee

\end{prop}

\begin{proof}
Given $u_a$, $u_b$ and $v_b$ well-defined, $v_a$ is uniquely defined from \eqref{G2}. We solve for $u_a$ now. 

\medskip

We follow the ideas of the proof  of Proposition \ref{IFT}. However, this time we consider $G$ defined in the opposite sense: using \eqref{uc},
\[
G = (G_3,G_4), \quad G=G(\widetilde u_c, u_b, v_a, v_b, m),
\]

\[
 \begin{array}{ccc} G:  H^2(\R;\Com) \times H^1(\R;\Com) \times H^{-1}(\R;\Com)^2 \times \Com &  \longrightarrow &H^1(\R;\Com) \times H^{-1}(\R;\Com) ,\\ 
  ( \widetilde u_c, u_b,v_a, v_b, m)   & \longmapsto & G( \widetilde u_c, u_b, v_a, v_b, m) \end{array} 
\]
with
\be\label{ortho0}
\int_\R (\widetilde u_c)_x   (\frac{1}{\mu^1})_x  =0,
\ee
where, from \eqref{G1},
\be\label{G3}
G_3(\widetilde u_c, u_b,v_a,v_b,m):=   \frac{(u_c^0 + u_c - 2u_b)}{\sqrt{2}}  - m \, \sin \Big(\frac{\widetilde u_c^0 + \widetilde u_c}{\sqrt{2}}\Big),
\ee
and from \eqref{G2},
\bea\label{G4}
& & G_4(\widetilde u_c, u_b, v_a, v_b,m)   :=   v_a- v_b   \nonu\\
& &  \qquad + \, m \,  \Big[ (u_c^0 + u_c)_{x} \cos \Big(\frac{\widetilde u_c^0 + \widetilde u_c}{\sqrt{2}} \Big)   + \frac { ( (u_c^0 + u_c -u_b)^2 + u_b^2)}{\sqrt{2}}\sin \Big( \frac{\widetilde u_c^0 +\widetilde u_c}{\sqrt{2}} \Big) \Big] . \nonu\\
& & 
\eea
Clearly  $G$ as in \eqref{G3}-\eqref{G4} defines a $C^1$ functional in a small neighborhood of $X^2$ given by 
\be\label{X2}
X^2:= (0,u_b^0, v_a^0, v_b^0, m^0) \in H^2(\R;\Com) \times H^1(\R;\Com) \times H^{-1}(\R;\Com)^2 \times \Com,
\ee
where $G$ is well-defined according to \eqref{Conds12} and $G(X^2)=(0,0)$.

\medskip

Fix $m$ close enough to $m^0$. Now we have to show that 
\be\label{ua}
u_c  - m \cos \Big(\frac{\widetilde u_c^0}{\sqrt{2}}\Big) \widetilde u_c = f,
\ee
has a unique solution $\widetilde u_c$ such that $u_c  \in H^2(\R;\Com)$, for any $ f \in H^1(\R;\Com)$. Indeed, consider $\mu^1$ given by \eqref{Mu2}. It is not difficult to check that  (see conditions  \eqref{rem0}, \eqref{Smallness} and \eqref{Conds13})
\[
\re m >0,
\]

\be\label{3p7}
\lim_{\pm \infty}\cos\Big(\frac{\widetilde u_c^0}{\sqrt{2}}\Big) =\mp 1,
\ee
and
\[
\mu^1 = \exp \Big(-m \int_{0}^x \cos\Big(\frac{\widetilde u_c^0}{\sqrt{2}}\Big)\Big).
\]
Note that from \eqref{3p7} and \eqref{rem0} we have that $|\mu^1(x)|$ is exponentially growing in space as $x\to \pm\infty.$  We have from \eqref{ua}, 
\[
(\mu^1 \widetilde u_c)_x = \mu^1 f ,
\]
so that, thanks to \eqref{Cond212}, 
\[
\widetilde u_c = \frac{1}{\mu^1} \mu^1(0)\widetilde u_c(0) + \frac{1}{\mu^1} \int_0^x \mu^1 f.
\]
Clearly $\lim_{\pm \infty} \widetilde u_c =0$ for $f \in H^1(\R;\Com)$. In order to ensure uniqueness, we ask for $\widetilde u_c$ satisfying
\[
\int_\R u_c   (\frac{1}{\mu^1})_x  =0,
\]
which is nothing but \eqref{ortho0} and \eqref{ortho}, which is justified by \eqref{Cond212}. Let us show that $\widetilde u_c \in L^2(\R;\Com)$. We have for $x>0$ large
\[
|\widetilde u_c(x)| \leq C  \int_0^x e^{-(\re m) (x-s)}|f(s)|ds = C e^{-(\re m) (\cdot) } \star |f|,  \qquad \re m>0. 
\]
A similar estimate can be established if $x<0$. Therefore, using Young's inequality,
\[
\|\widetilde u_c\|_{L^2(\R;\Com)} \leq C \|f\|_{L^2(\R;\Com)},
\]
as desired. Now we check that $u_c \in H^1(\R;\Com).$ Indeed, we have
\[
u_c =  f -  \frac{\mu^1_x}{(\mu^1)^2} \int_0^x \mu^1 f.
\]
Since $\mu^1_x / \mu^1$ is bounded, we have proven that $u_c \in L^2(\R;\Com)$. A new iteration proves that $u_c \in H^1(\R;\Com)$. Estimates \eqref{Conds234}-\eqref{Conds23a}-\eqref{Cua} are consequence of the Implicit Function Theorem and can be proved as in the previous Proposition. The proof is complete. 
\end{proof}

We finish this Section by pointing out the role played by the B\"acklund transformation in the mKdV dynamics. We recall the following standard result.

\begin{thm}\label{Cauchy1}
Let $m\in \Com$ be a {\bf fixed} parameter, and $I \subset\R$ an open time interval. Assume that $u_b \in C(I; H^1(\R;\Com))$ and solves \eqref{mKdV}, i.e., 
\be\label{ecu0}
(u_b)_t +((u_b)_{xx} +u_b^3)_x=0,
\ee
in the $H^1$-sense. Define $v_b := \int_0^x (u_b)_t $ as  a distribution in $H^{-1}(\R;\Com)$. Then for each $t\in I$ the corresponding solution $(u_a(t), v_a(t))$ of \eqref{G1}-\eqref{G2} for $m$ fixed, obtained in the space $ H^1(\R;\Com) \times H^{-1}(\R;\Com)$, satisfies the following:  

\ben
\item $u_a\in C(I; H^1(\R;\Com))$; 
\item $(u_a)_t =(v_a)_x$ is well-defined in $H^{-2}(\R;\Com)$; and
\item  $u_a$ solves \eqref{mKdV} in the $H^1$-sense.
\een

\end{thm}

\begin{proof}
The first step is an easy consequence of the continuous character of the solution map given by the implicit function theorem.
By density we can assume $u_b(t) \in H^3(\R;\Com)$. From \eqref{G1} we have
\be\label{Pivote0}
(u_a)_x- (u_b)_x = m\cos\Big(\frac{\widetilde u_a+\widetilde u_b}{\sqrt{2}}\Big)(u_a+u_b),
\ee
and
\[
(u_a)_{xx}-(u_b)_{xx} = m\cos\Big(\frac{\widetilde u_a+\widetilde u_b}{\sqrt{2}}\Big)((u_a)_x+(u_b)_x)  -\frac m{\sqrt{2}}\sin\Big(\frac{\widetilde u_a+\widetilde u_b}{\sqrt{2}}\Big)(u_a+u_b)^2.
\]
Therefore, from \eqref{G2} and \eqref{G1},
\bee
v_a-  v_b & =&    -  ((u_a)_{xx}-(u_b)_{xx} )    -  \frac m{\sqrt{2}}\sin\Big(\frac{\widetilde u_a+\widetilde u_b}{\sqrt{2}}\Big)  [(u_a+u_b)^2  + (u_a^2 + u_b^2) ]  \\
&  = &  -  ((u_a)_{xx}-(u_b)_{xx} ) -  (u_a-u_b) (u_a^2+u_a u_b+u_b^2)\\
& =& -((u_a)_{xx} +u_a^3) -((u_b)_{xx} +u_b^3).
\eee
Since $(v_b)_x = (u_b)_t$, we have from \eqref{ecu0} that $(v_b)_x  +((u_b)_{xx} +u_b^3)_x =0$. Therefore
\be\label{Pivote}
(v_a)_x +((u_a)_{xx} +u_a^3)_x =0.
\ee
Finally, if $(u_a)_t = (v_a)_x$, we have that  $u_a$ solves \eqref{mKdV}. In order to prove this result, we compute the time derivative in \eqref{G1}: we get
\be\label{Pivote1}
(u_a)_t-(u_b)_t  = m \cos \Big(\frac{\widetilde u_a+\widetilde u_b}{\sqrt{2}}\Big)((\widetilde u_a)_t+(\widetilde u_b)_t).
\ee
Note that given $u_b$, the solution $u_a$ is uniquely defined, thanks to the Implicit Function Theorem.
Additionally, from \eqref{G2}
\bee
&& (v_a)_x-(v_b)_x + \\
& &  \quad + \, m \,  \Big[ ( (u_a)_{xx} +  (u_b)_{xx}) \cos \Big(\frac{\widetilde u_a+\widetilde u_b}{\sqrt{2}} \Big) \\
& & -  \frac 1{\sqrt{2}}( (u_a)_{x} +  (u_b)_{x})  (u_a+u_b) \sin \Big(\frac{\widetilde u_a+\widetilde u_b}{\sqrt{2}} \Big) \\
& & \quad   +  \sqrt{2}  (u_a (u_a)_x + u_b (u_b)_x) \sin \Big( \frac{\widetilde u_a+\widetilde u_b}{\sqrt{2}} \Big)\\
&&  + \frac{ (u_a^2 + u_b^2)}{2}(u_a+u_b)\cos \Big( \frac{\widetilde u_a+\widetilde u_b}{\sqrt{2}} \Big) \Big]  =0.
\eee
We use \eqref{G1} and \eqref{G2} in the previous identity we get
\bee
&& (v_a)_x-(v_b)_x + \\
& &  \quad +  \Big[ m ( (u_a)_{xx} +  (u_b)_{xx}) \cos \Big(\frac{\widetilde u_a+\widetilde u_b}{\sqrt{2}} \Big) + (u_a^2 -u_au_b +u_b^2) ((u_a)_x-(u_b)_x)  \Big]  =0.
\eee
Finally, we use \eqref{Pivote0} to obtain
\[
 (v_a)_x-(v_b)_x +m\cos \Big(\frac{\widetilde u_a+\widetilde u_b}{\sqrt{2}} \Big) ((u_a)_{xx} +u_a^3 +(u_b)_{xx} + u_b^3) =0,
\]
so \eqref{ecu0} and \eqref{Pivote} imply
\[
(v_a)_x-(v_b)_x = m\cos \Big(\frac{\widetilde u_a+\widetilde u_b}{\sqrt{2}} \Big) ( v_a +v_b),
\]
so that from \eqref{Pivote1} and the uniqueness we conclude.
\end{proof}

\bigskip

\section{Dynamics of complex-valued mKdV solitons}\label{4}

\medskip

In what follows we will apply the results from the previous section in a neighborhood of the complex soliton at time equals zero. 
Define (cf. \eqref{tQ}),
\be\label{Q00}
\widetilde Q^0 := \widetilde Q(x;\al,\bt, 0,0),
\ee
and similarly for $Q^0$ and $\widetilde Q_t^0$. Recall that from Lemma \ref{Good} the complex soliton $Q^0$ is everywhere well-defined since  \eqref{BadCondFF} is not satisfied.  Finally, given any 
\[
\widetilde z_b^0 \in \dot H^1(\R;\Com),
\]
we define $z_b^0$ by the identity (see \eqref{tildeN} for instance)
\[
 z_b^0 := (\widetilde z_b^0)_x,
\]
and in term of distributions,
\[
w_b^0:= -( (z_b^0)_{xx} + (z_b^0)^3) \in H^{-1}(\R;\Com).
\]

\begin{lem}\label{L2}
There exists $\nu_0>0$ and $C>0$ such that, for all $0<\nu<\nu_0$, the following holds. For all $z_b^0 \in H^1(\R;\Com)$ satisfying
\be\label{smallNU}
\|z_b^0\|_{H^1(\R;\Com)} <\nu,
\ee
there exist unique $ y_a^0\in H^1(\R,\Com)$, $ y_a^1\in H^{-1}(\R,\Com)$ and $ m \in \Com$, of the form 
\be\label{deco}
y_a^0(x)=  y_a^0[z_b^0](x) ,\quad y_a^1(x)=  y_a^1[z_b^0, w_b^0](x), \quad  m :=  \bt+ i\al + q^0,
\ee
such that 
\[
\|y_a^0\|_{H^1(\R;\Com)} +| q^0| \leq C\nu,
\]

\[
\widetilde z_a^0 + \widetilde y_a^0 \in H^2(\R;\Com), \footnote{Note that both $\widetilde z_a^0$  and $\widetilde y_a^0$ may be unbounded functions, but the addition is bounded on $\R$.}
\]
and
\be\label{Goo}
G(Q^0 + z_a^0, y_a^0,  \widetilde Q_t^0 +  w_a^0,  y_a^1 ,  m) \equiv (0,0).
\ee
\end{lem}

\begin{proof}
Let $Q^0$ be the soliton profile with parameters $\bt,\al$ and $x_1=x_2=0$ (cf. \eqref{Q00}). We apply Proposition \eqref{IFT} with 
\[
u_a^0 := Q^0, \quad u_b^0 := 0,  \quad v_a^0 := \widetilde Q_t^0, \quad v_b^0:=0,
\]
and
\[
  m^0 := \bt+ i\al. 
\]
Clearly $\widetilde u_a^0 + \widetilde u_b^0 = \widetilde Q_0 $ satisfies \eqref{Conds12}-\eqref{Conds13}.
 From \eqref{zero0} we have
\be\label{ecuQQ}
(Q^0)_x  -(\bt+i\al) \cos \Big(\frac{\widetilde Q^0}{\sqrt{2}} \Big) Q^0=0, \quad Q^0(-\infty) =0,
\ee
so that we have (cf. \eqref{Mu0}-\eqref{Cond110})
\[
\mu^0 = Q^0.
\]
Clearly $Q^0$ is never zero. Moreover, $|(Q^0)^{-1}Q^0_x|$ is bounded on $\R$. Now we prove that
\[
\int_\R \sin\Big(\frac{\widetilde Q^0}{\sqrt{2}}\Big)Q^0  \neq 0.
\]
From \eqref{zero0} and \eqref{MassQ},
\[
\int_\R \sin\Big(\frac{\widetilde Q^0}{\sqrt{2}} \Big)Q^0 = \frac1{\sqrt{2}(\bt+i\al)}\int_\R (Q^0)^2 =  \frac{4(\bt+i\al)}{\sqrt{2}(\bt+i\al)} = 2\sqrt{2}.
\]
\end{proof}

Before continuing, we need some definitions.  We denote
\be\label{ABStar}
\al^*:= \al + \ima q^0, \qquad \bt^* := \bt + \re q^0,
\ee
such that $m$ in \eqref{deco} satisfies
\[
m = \bt+ i\al + q^0 = \bt^* + i \al^*.
\]
Since $q^0$ is small, we have that $\bt^*$ and $\al^*$ are positive quantities. Similarly, define
\be\label{DGStar}
\delta^* := (\al^*)^2 -3(\bt^*)^2, \quad \ga^* := 3(\al^*)^2 -(\bt^*)^2,
\ee
and compare with \eqref{deltagamma}.

\medskip

Consider the profile $\widetilde Q$ introduced in \eqref{tQ}. We consider,  for all $t \in \R$, the complex soliton profile
\be\label{dinatQ}
\widetilde Q^*(t,x):= \widetilde Q(x;\al^*,\bt^*, \delta^* t +x_1, \ga^* t+x_2),
\ee
with $\delta^*$ and $\ga^*$ defined in \eqref{DGStar}, $x_1$ and $x_2$ \emph{possibly depending on time}, and
\be\label{dinaQ}
Q^*(t,x):= \partial_x \widetilde Q^*(t,x).
\ee
It is not difficult to see that (see e.g.  \eqref{Qc})
\[
Q^*(t,x) = Q_{c}(x-ct -\hat x), \quad \sqrt{c} =\bt^* + i\al^*, \quad \hat x  \in \Com,
\]
which is a \emph{complex-valued} solution of mKdV \eqref{mKdV}.  Technically, the complex soliton $Q^*(t)$ has velocity $-\ga^* = (\bt^*)^2 - 3(\al^*)^2$, a quantity that is always smaller than the corresponding speed $(= (\bt^*)^2)$ of  the associated real-valued soliton $Q_{(\bt^*)^2}$ obtained by sending $\al^*$ to zero. Finally, as in \eqref{tQt0}  we define
\[
\widetilde Q^*_t(t,x) :=  -(\bt^*+i\al^*)^2 Q^*(t,x).
\]

\begin{lem}\label{4p2}
Fix $\al,\bt>0$. Assume that $x_1,x_2$ are time dependent functions such that 
\be\label{sm}
|x_1'(t)| +|x_2'(t)| \ll |\delta^* -\ga^*| =2((\al^*)^2 +(\bt^*)^2).
\ee
Then there exists only a sequence of times $t_k \in \R$, $k\in \Z$ such that \eqref{BadCondFF} is satisfied.
In particular, $(t_k)$ is a sequence with no limit points. 
\end{lem}
\begin{proof}
Note that \eqref{BadCondFF} reads now
\[
(\delta^* -\ga^*)t_k +(x_1-x_2)(t_k) = \frac{\pi}{\al^*}\Big(k+\frac 12 \Big).
\]
Since from \eqref{DGStar}  $\delta^* -\ga^*  = -2((\al^*)^2 +(\bt^*)^2) \neq 0,$ and using \eqref{sm} and the Inverse Function Theorem applied for each $k$, we have the desired conclusion.
\end{proof}

We conclude that $\widetilde Q^*$ and $ Q^*$  defined in \eqref{dinatQ} and \eqref{dinaQ} are well-defined except for an isolated sequence of times $t_k$. We impose now the following condition:

\be\label{BadCond}
 t\in \R \hbox{ satisfies } t\neq t_k \hbox{ for all } k\in \Z.
\ee
\smallskip

In what follows we will solve the Cauchy problem associated to mKdV with suitable initial data.  Indeed, we will assume that 
\be\label{NiceCond}
\hbox{ $y_a^0$ is a real-valued function,  and $y_a^0 \in H^1(\R)$.}
\ee
We will need the following\footnote{We recall that this result is consequence of the local Cauchy theory and the conservation of mass and energy \eqref{M1}-\eqref{E1}.}

\begin{thm}[Kenig, Ponce and Vega \cite{KPV}]\label{KPV}
For any $y_a^0\in H^1(\R)$, there exists a unique\footnote{In a certain sense, see \cite{KPV}.} solution $y_a\in C(\R,H^1(\R))$ to mKdV, and
\[
\sup_{t\in \R}\|y_a(t)\|_{H^1(\R)} \leq C \|y_a^0\|_{H^1(\R)},
\]
with $C>0$ independent of time. Moreover, the \emph{mass}
\be\label{M1}
M[y_a](t)  :=  \frac 12 \int_\R y_a^2(t,x)dx = M[y_a^0],
\ee 
and \emph{energy} 
\be\label{E1}
E[y_a](t)  :=  \frac 12 \int_\R (y_a)_x^2(t,x)dx -\frac 14 \int_\R (y_a)^4(t,x)dx = E[y_a^0]
\ee
are conserved quantities.
\end{thm}

\medskip

Assume \eqref{NiceCond}. Let $y_a \in C(\R,H^1(\R))$  denote the corresponding solution for mKdV with initial data $y_a^0$.  Since $\|y_a^0\|_{H^1} \leq C \eta$, we have for a (maybe different) constant $C>0$,
\be\label{MaMa}
\sup_{t\in \R}\|y_a(t)\|_{H^1(\R)} \leq C\eta.
\ee
In particular,  we can define, for all $t\in \R$,
\[
\widetilde y_a(t) := \int_0^x y_a(t,s)ds,
\]
and
\be\label{yt}
 (\widetilde y_a)_t (t) := -((y_a)_{xx}(t) + y_a^3(t)) \in H^{-1}(\R),
\ee
because $y_a(t) \in L^p(\R)$ for all $p\geq 2.$

\begin{lem}\label{L3}
Assume that a time $t \in \R$ and $y_a^0$ are such that \eqref{BadCond} and \eqref{NiceCond} hold.  Then there are unique $z_b=z_b(t) \in H^1(\R;\Com) $  and $w_b=w_b(t) \in H^{-1}(\R;\Com)$ such that for all $t\neq t_k$,

\be\label{H1ba}
\widetilde z_b+\widetilde y_a \in H^2(\R;\Com),
\ee

\be\label{312}
\frac 1{\sqrt{2}} (Q^* + z_b  -y_a) =(\bt + i\al + q^0) \sin\Big( \frac{\widetilde Q^*+\widetilde z_b+\widetilde y_a}{\sqrt{2}}\Big),
\ee
where $\widetilde Q$ and $Q$ are defined in \eqref{dinatQ} and \eqref{dinaQ}. Moreover, we have
\bea\label{412}
& & 0= \widetilde Q_t^*  + w_b - (\widetilde y_a)_t    \nonu\\
& &  \qquad + \, (\bt + i\al + q^0)  \Big[ (   Q_x^* + (z_b)_x + (y_a)_x) \cos \Big(\frac{\widetilde Q^* + \widetilde z_b +\widetilde y_a}{\sqrt{2}} \Big) \nonu\\
& & \qquad  \qquad \qquad \qquad \quad   + \frac { ( (Q^*+z_b)^2 + y_a^2)}{\sqrt{2}}\sin \Big( \frac{\widetilde Q^* + \widetilde z_b + \widetilde y_a}{\sqrt{2}}\Big) \Big],
\eea
and for all $t\neq t_k$,
\be\label{Cua2}
\|z_b(t)\|_{H^1(\R;\Com)}  <C \nu,
\ee
with $C$ divergent as $t$ approaches a $t_k$.
\end{lem}

\begin{proof}
We will use Proposition \ref{IFT2}.  For that it is enough to recall that from \eqref{zero0} and \eqref{zero1}, and for all $t\neq t_k,$\footnote{It is interesting to note that the shifts $x_1,x_2$ on $Q^*(t,x)$ cannot be modified, otherwise there is no continuity at $t=0$.}
\be\label{GQStar}
\frac 1{\sqrt{2}} Q^* =(\bt + i\al + q^0) \sin\Big( \frac{\widetilde Q^*}{\sqrt{2}}\Big),
\ee
and
\[
 \widetilde Q_t ^*   + (\bt+ i\al +q^0)  \Big[  Q_x^*  \cos \Big(\frac{\widetilde Q^*}{\sqrt{2}} \Big)   + \frac { (Q^*)^2}{\sqrt{2}}\sin \Big( \frac{\widetilde Q^*}{\sqrt{2}} \Big) \Big]  =0,
\]
so that we can apply Proposition \ref{IFT2} at $X^0 =(Q^*, 0, \widetilde Q^*_t,0,m)$, where from \eqref{GQStar} we have  $m = (\bt + i\al + q^0)$. It is not difficult to see that the function $\mu^1$ in \eqref{Mu2} is given by
\[
\mu^1 = (Q^*)^{-1},
\]
and \eqref{Cond212} is satisfied. Note that we require the estimate \eqref{MaMa} in order to obtain \eqref{312}-\eqref{412}.  Finally, \eqref{Cua2} is a direct consequence of \eqref{Cua} and the fact that $\norm{\mu^1}_{H^1(\R;\Com)}$ blows up as $t$ approaches some $t_k$.
\end{proof}

\begin{rem}
Since from \eqref{Goo} we get
\[
\frac 1{\sqrt{2}} (Q^0 + z_b^0 -y_a^0)  =(\bt + i\al + q^0) \sin\Big( \frac{\widetilde{Q}^0 + \widetilde{z_b}^0 + \widetilde{y_a}^0}{\sqrt{2}}\Big),
\]
we have that \eqref{312} implies by uniqueness
\[
(Q^* + z_b-y_a)(t=0) = Q^0 + z_b^0 -y_a^0,
\]
i.e.,
\[
(Q^* + z_b)(t=0) = Q^0 + z_b^0 .
\]
\end{rem}

We are ready to prove a detailed version of Theorem \ref{T2}, a result on complex-valued solitons.

\begin{thm}\label{T2a}
There exists $\nu_0>0$ such that for all $0<\nu<\nu_0$ the following holds. Consider the initial data $u_b^0:= Q^0 + z_b^0 \in H^1(\R;\Com)$, where
\[
\|z_b^0\|_{H^1(\R;\Com)} < \nu.
\] 
Assume in addition that the corresponding function $y_a^0$ given by Lemma \ref{L2} is real-valued and belongs to $H^1(\R)$. Fix $\ve_0>0$. Then for all $t$  such that $|t-t_k|\geq \ve_0$, with $t_k$ defined in Lemma \ref{4p2}, the function  $u_b = Q^*+z_b$ introduced in Lemma \ref{L3}  is an $H^1$ complex valued solution of mKdV, it satisfies $(u_b)_t = (Q^*+z_b)_t = (\tilde Q_t ^*+ w_b)_x$ and the following estimate holds
\be\label{ComplexStab}
\sup_{|t-t_k|\geq \ve_0}\| u_b(t) - Q^*(t)\|_{H^1(\R;\Com)} \leq C_{\ve_0} \nu.
\ee
\end{thm}

\begin{rem}
The quantity $\ve_0>0$ is just an auxiliary parameter and it can be made as small as required; however the constant $C_{\ve_0}$ in \eqref{ComplexStab} becomes singular as $\ve_0$ approaches zero.
\end{rem}

\begin{rem}
In Corollary \ref{RealY} we will prove that there is an open set in $H^1(\R;\Com)$ leading to $y_a^0$ real-valued. The openness of this set will be a consequence of the Implicit Function Theorem.
\end{rem}

\begin{proof}
We apply Lemma \ref{L2}. Assuming \eqref{NiceCond} we have $y_a^0$ real-valued so that there is an mKdV dynamics $y_a(t)$  constructed in Theorem \ref{KPV}. Finally we apply Lemma \ref{L3} to obtain the dynamical function $Q^*(t)+z_b(t)$. Finally we conclude using Theorem \ref{Cauchy1}.
\end{proof}

Now we will prove that the mass and energy
\be\label{massQ}
\frac 12\int_\R u_b^2(t), \qquad \frac 12 \int_\R (u_b)_x^2(t) -\frac 14 \int_\R u_b^4(t),
\ee
remain conserved for all time, without using the mKdV equation \eqref{mKdV}, but only the B\"acklund transformation \eqref{312}. The fact that $\widetilde z_b + \widetilde y_a \in H^1(\R;\Com)$ will be essential for the proof.

\begin{cor}\label{CLaw0}
Assume that $t\neq t_k$  for all $k\in \Z$. Then the quantity
\be\label{Ident1}
\frac 12 \int_\R (Q^*+ z_b )^2(t)
\ee
is well-defined and independent of time, and
\be\label{Ident2}
\frac 12 \int_\R (Q^*+z_b)^2(t) = \frac 12 \int_\R (y^0_a)^2 + 2(\bt + i\al + q^0).
\ee
Moreover, \eqref{Ident2} can be extended in a continuous form to every $t\in \R.$ 
\end{cor}

\begin{proof}
Using \eqref{312} and multiplying each side by $ \frac 1{\sqrt{2}}(Q^* + z_b + y_a)$ we obtain
\[
\frac1{2} (Q^* +z_b -y_a)(Q^* + z_b - y_a) = -(\bt +i\al +q^0)  \Big[\cos \Big(\frac{\widetilde Q^* +\widetilde z_b+ \widetilde y_a}{\sqrt{2}}\Big)\Big]_x .
\]
We integrate on $\R$ to obtain
\[
\frac1{2}  \int_\R(Q^* +z_b -y_a)(Q^*+ z_b - y_a) =- (\bt +i\al +q^0) \cos \Big(\frac{\widetilde Q^* + \widetilde z_b +\widetilde y_a}{\sqrt{2}}\Big) \Bigg|_{-\infty}^\infty .
\]
Since $\lim_{-\infty} \widetilde Q^* =0$, $\lim_{+\infty} \widetilde Q^* = \sqrt{2}\pi$ (see \eqref{tQinfty}) and $\lim_{\pm \infty}( \widetilde z_b +\widetilde y_a) =0$,  we get \eqref{Ident1}-\eqref{Ident2}, because the mass of $y_a(t)$ is conserved.
\end{proof}

\begin{cor}\label{CLaw3}
Assume that $t\neq t_k$  for all $k\in \Z$. Then the quantity
\be\label{Ident5}
E[Q^*+ z_b] (t):= \frac 12 \int_\R (Q^*+ z_b )_x^2(t)  -\frac 14 \int_\R (Q^*+ z_b )^4(t) 
\ee
is well-defined and independent of time. Moreover, it satisfies
\[
\frac 12\int_\R  (Q^* + z_b)^2_x(t)  -\frac14 \int_\R (Q^* + z_b)^4(t)   =E[y_a^0]  - \frac23 (\bt^* + i\al^*)^3.
\]
Finally, this quantity can be extended in a continuous form to every $t\in \R.$ 
\end{cor}

\begin{proof}
Denote $m=(\bt + i\al + q^0). $ From \eqref{312} we have
\be\label{312x}
 (Q^* + z_b)_x  -(y_a)_x =m \cos\Big( \frac{\widetilde Q^*+\widetilde z_b+\widetilde y_a}{\sqrt{2}}\Big) (Q^* + z_b +y_a).
\ee
Multiplying by $ (Q^* + z_b)_x  + (y_a)_x$ we get
\bea
 (Q^* + z_b)^2_x  -(y_a)^2_x  & = &m \cos\Big( \frac{\widetilde Q^*+\widetilde z_b+\widetilde y_a}{\sqrt{2}}\Big) (Q^* + z_b +y_a) (Q^* + z_b +y_a)_x \nonu \\
 & =& m \cos\Big( \frac{\widetilde Q^*+\widetilde z_b+\widetilde y_a}{\sqrt{2}}\Big)  \Big[  (Q^* + z_b )   (Q^* + z_b )_x + y_a (y_a)_x   \nonu \\
 & & \qquad \qquad \qquad  y_a  (Q^* + z_b )_x + (Q^* + z_b ) (y_a)_x \Big]. \label{Ene1}
\eea
On the other hand, we multiply \eqref{312x} by $y_a$ and $(Q^* + z_b )$ to obtain
\[
y_a (Q^* + z_b)_x  - y_a (y_a)_x =m \cos\Big( \frac{\widetilde Q^*+\widetilde z_b+\widetilde y_a}{\sqrt{2}}\Big) (Q^* + z_b +y_a) y_a,
\]
and
\[
(Q^* + z_b )(Q^* + z_b)_x  -  (Q^* + z_b )(y_a)_x =m \cos\Big( \frac{\widetilde Q^*+\widetilde z_b+\widetilde y_a}{\sqrt{2}}\Big) (Q^* + z_b +y_a)(Q^* + z_b ).
\]
If we subtract the latter from the former we get
\bea
y_a (Q^* + z_b)_x + (Q^* + z_b ) (y_a)_x  &=&  (Q^* + z_b )(Q^* + z_b)_x + y_a (y_a)_x \nonu  \\
& &  + m \cos\Big( \frac{\widetilde Q^*+\widetilde z_b+\widetilde y_a}{\sqrt{2}}\Big) [ y_a^2 -(Q^* + z_b )^2 ]. \nonu\\
& & \label{Ene2}
\eea
Replacing \eqref{Ene2} into \eqref{Ene1},
\bea
 (Q^* + z_b)^2_x  -(y_a)^2_x  & = & m \cos\Big( \frac{\widetilde Q^*+\widetilde z_b+\widetilde y_a}{\sqrt{2}}\Big) [ (Q^* + z_b )^2 + y_a^2]_x   \nonu \\
 & & +  m^2  \cos^2 \Big( \frac{\widetilde Q^*+\widetilde z_b+\widetilde y_a}{\sqrt{2}}\Big) [  y_a^2 -(Q^* + z_b )^2  ]. \label{Ene3}
\eea
Finally we use once again \eqref{312}. We multiply by $(Q^* + z_b  + y_a)$: 
\[
\frac 1{\sqrt{2}} [(Q^* + z_b)^2  -y_a^2] =m \sin\Big( \frac{\widetilde Q^*+\widetilde z_b+\widetilde y_a}{\sqrt{2}}\Big)(Q^* + z_b  + y_a).
\]
Replacing in \eqref{Ene3} we finally arrive to the following identity:
\bea
& &  (Q^* + z_b)^2_x  -(y_a)^2_x   = \nonu\\
& & \qquad =  m \cos\Big( \frac{\widetilde Q^*+\widetilde z_b+\widetilde y_a}{\sqrt{2}}\Big) [ (Q^* + z_b )^2 + y_a^2]_x   \nonu \\
 & & \qquad \quad   -  m^3\sqrt{2}  \cos^2 \Big( \frac{\widetilde Q^*+\widetilde z_b+\widetilde y_a}{\sqrt{2}}\Big) \sin\Big( \frac{\widetilde Q^*+\widetilde z_b+\widetilde y_a}{\sqrt{2}}\Big)(Q^* + z_b  + y_a).\nonu
\eea
The last term on the right hand side above can be recognized as a total derivative. After integration, we obtain
\bee
\int_\R [ (Q^* + z_b)^2_x  -(y_a)^2_x ] & = & m \int_\R \cos\Big( \frac{\widetilde Q^*+\widetilde z_b+\widetilde y_a}{\sqrt{2}}\Big) [ (Q^* + z_b )^2 + y_a^2]_x   \nonu \\
 & &  + \frac23 m^3  \cos^3 \Big( \frac{\widetilde Q^*+\widetilde z_b+\widetilde y_a}{\sqrt{2}}\Big)\Big|_{-\infty}^{+\infty} \\
 & =& \frac m{\sqrt{2}} \int_\R \sin\Big( \frac{\widetilde Q^*+\widetilde z_b+\widetilde y_a}{\sqrt{2}}\Big) (Q^* + z_b  + y_a) [ (Q^* + z_b )^2 + y_a^2]   \nonu \\
 & &   - \frac43 m^3  \\
 & =& \frac 1{2} \int_\R [ (Q^* + z_b)^2  - y_a^2] [ (Q^* + z_b )^2 + y_a^2]  - \frac43 m^3 \\
 & =& \frac12 \int_\R [(Q^* + z_b)^4  - y_a^4 ] - \frac43 m^3.
\eee
Finally
\[
\frac 12\int_\R  (Q^* + z_b)^2_x  -\frac14 \int_\R (Q^* + z_b)^4   = \frac 12\int_\R (y_a)^2_x  -\frac 14 \int_\R y_a^4  - \frac23 (\bt + i\al + q^0)^3.
\]
Since the right hand side above is conserved for all time, we have proved \eqref{Ident5}.
\end{proof}

\bigskip

\section{Complex solitons versus breathers}\label{5}

\medskip

We introduce now the notion of breather profile. Given parameters  $x_1 ,x_2 \in \R$ and $\al, \bt>0$, we consider $y_1$ and $y_2$ defined in \eqref{y1y2mod}. Let $\tilde B$ be the kink profile
\be\label{tB}
\widetilde B = \widetilde B(x; \al,\bt, x_1, x_2) := 2\sqrt{2} \arctan \Big( \frac{\bt}{\al} \frac{\sin(\al y_1)}{\cosh(\bt y_2)}\Big),
\ee
and with a slight abuse of notation, we redefine
\be\label{BBB}
B:=  \widetilde B_x.
\ee
Note that
\be\label{tBinfty}
\widetilde B(-\infty) =\widetilde B(+\infty) =0,
\ee
and for $k\in \Z,$
\be\label{perio2}
\begin{cases}
\displaystyle{\widetilde B(x; \al,\bt, x_1 + \frac{k \pi}\al, x_2) =(-1)^k \widetilde B(x; \al,\bt, x_1, x_2), }\\
\hbox{\;}\\
\displaystyle{B(x; \al,\bt, x_1 + \frac{k \pi}\al, x_2) =(-1)^k  B(x; \al,\bt, x_1 , x_2).}
\end{cases}
\ee

\medskip

Now we introduce the directions associated to the shifts $ x_1$ and $ x_2$. Given a breather profile of parameters $ \al$, $\bt$, $ x_1$ and $ x_2$, we define
\be\label{B1}
B_1 = B_1 (x; \al, \bt,  x_1, x_2)  := \partial_{x_1} B ,
\ee
\be\label{B2}
B_2 =B_2 (x; \al, \bt,  x_1, x_2)  := \partial_{ x_2}  B .
\ee
and for $\delta$ and $\gamma$ defined in \eqref{deltagamma}, 
\be\label{tBt}
\widetilde B_t := \delta B_1 +\ga B_2.
\ee
We also have
\be\label{tBteq}
\widetilde B_t + B_{xx} + B^3 =0,
\ee
see \cite{AM} for a proof of this identity. 

\medskip

If $ x_1$ or $ x_2$ are time-dependent variables, we assume that the associated $B_j$ corresponds to the partial derivative with respect to the time-independent variable $ x_j$, evaluated at $ x_j(t)$.

\medskip

In this Section we will prove that there is a deep interplay between complex solitons and  breather profiles. We start with the following identities.

\begin{lem}\label{Equal}
Let $(B,Q)$ be a pair breather-soliton profiles with scaling parameters $\al,\bt>0$ and shifts $x_1,x_2\in \R$. Assume that \eqref{BadCondFF} is not satisfied. Then we have
\be\label{zerot}
\frac{(B-Q)}{\sqrt{2}}  - (\bt-i\al) \sin \Big(\frac{\widetilde B+\widetilde Q}{\sqrt{2}}\Big) \equiv 0 ,
\ee
and
\be\label{zerot2}
\widetilde B_t -\widetilde Q_t    + (\bt-i\al)  \Big[ ( B_x + Q_x) \cos \Big(\frac{\widetilde B+\widetilde Q}{\sqrt{2}} \Big)   + \frac { (B^2 + Q^2)}{\sqrt{2}}\sin \Big( \frac{\widetilde B+\widetilde Q}{\sqrt{2}} \Big) \Big] \equiv 0.
\ee
\end{lem}

\begin{proof}
Let us assume \eqref{zerot} and prove \eqref{zerot2}. We have from \eqref{tQt0} and \eqref{ecQ}
\[
\widetilde Q_t = -(\bt+i\al)^2 Q = -(Q_{xx} + Q^3).
\]
Using \eqref{zerot} we have
\[
B_x-Q_x  - (\bt-i\al)(B+Q) \cos \Big(\frac{\widetilde B+\widetilde Q}{\sqrt{2}}\Big) = 0 ,
\] 
and
\bee
& & B_{xx}-Q_{xx}  - (\bt-i\al) (B_x+Q_x) \cos \Big(\frac{\widetilde B+\widetilde Q}{\sqrt{2}}\Big)\\
& & \qquad + (\bt-i\al) \frac{(B+Q)^2}{\sqrt{2}}\sin \Big(\frac{\widetilde B+\widetilde Q}{\sqrt{2}}\Big) = 0 ,
\eee
so that using once again \eqref{zerot} and \eqref{tBteq} 
\bee
& & \widetilde B_t -\widetilde Q_t    + (\bt-i\al)  \Big[ ( B_x + Q_x) \cos \Big(\frac{\widetilde B+\widetilde Q}{\sqrt{2}} \Big)  
+ \frac { (B^2 + Q^2)}{\sqrt{2}}\sin \Big( \frac{\widetilde B+\widetilde Q}{\sqrt{2}} \Big) \Big] \\
& & = -(B_{xx} + B^3) + Q_{xx} + Q^3   +\Big[  B_{xx}-Q_{xx} + (\bt-i\al) \frac{(B+Q)^2}{\sqrt{2}}\sin \Big(\frac{\widetilde B+\widetilde Q}{\sqrt{2}}\Big) \Big] \\
& &  \qquad +  (\bt-i\al) \frac { (B^2 + Q^2)}{\sqrt{2}}\sin \Big( \frac{\widetilde B+\widetilde Q}{\sqrt{2}} \Big) \\
& & = Q^3-B^3 + \sqrt{2}(\bt-i\al)(B^2 +Q^2 + BQ) \sin \Big(\frac{\widetilde B+\widetilde Q}{\sqrt{2}}\Big) \\
& & =  Q^3-B^3 + (B^2 +Q^2 + BQ)(B-Q)  =0.
\eee 

\smallskip

The proof of \eqref{zerot} is a tedious but straightforward computation which deeply requires the nature of the breather and soliton profiles.  For the proof of this result, see Appendix \ref{AppA}.
\end{proof}

\begin{cor}\label{Cor0a}
Under the assumptions of Lemma \ref{Equal}, for any $x \in \R$ one has
\[
\frac{(B- \overline{Q})}{\sqrt{2}}  - (\bt + i\al) \sin \Big(\frac{\widetilde B+\widetilde{\overline Q}}{\sqrt{2}}\Big) \equiv 0  \quad \hbox{ in } \R,
\]
where $\overline{Q}$ is the complex-valued soliton with parameters $\bt$ and $-\al$.
\end{cor}

In order to prove some results in the next Section, we need several additional identities. 

\begin{cor}\label{Cor1}
Under the assumptions of Lemma \ref{Equal}, for any $x \in \R$ one has
\[
 \cos \Big(\frac{\widetilde B+\widetilde Q}{\sqrt{2}}\Big) = 1 - \frac1{2(\bt-i\al)} \int_{-\infty}^x (B^2-Q^2) ,
\]
and
\[
\lim_{x\to \pm \infty}\cos \Big(\frac{\widetilde B+\widetilde Q}{\sqrt{2}}\Big)(x) =  \mp 1.
\]
\end{cor}
\begin{proof}
We multiply by $\frac1{\sqrt{2}} (B + Q)$ in \eqref{zerot}. We get
\[
\frac1{2} (B^2-Q^2) -  (\bt-i\al) \sin \Big(\frac{\widetilde B+\widetilde Q}{\sqrt{2}}\Big) \times \frac1{\sqrt{2}} (B + Q) =0,
\]
i.e.,
\[
\frac1{2} (B^2-Q^2) +  (\bt-i\al) \partial_x \cos \Big(\frac{\widetilde B+\widetilde Q}{\sqrt{2}}\Big) =0.
\]
Since from \eqref{tQ} and \eqref{tB} one has
\[ 
\lim_{x\to -\infty} \cos \Big(\frac{\widetilde B+\widetilde Q}{\sqrt{2}}\Big) = 1.
\]
Therefore, after integration,
\[
\frac1{2} \int_{-\infty}^x (B^2-Q^2) +  (\bt-i\al) \Big[  \cos \Big(\frac{\widetilde B+\widetilde Q}{\sqrt{2}}\Big)  -1\Big] =0,
\]
as desired.
\end{proof}

\begin{lem}
We have
\bee
 \mathcal M  &:= & \frac 12\int_{-\infty}^x  B^2 \nonu\\
 & =& 2\bt \Big[ 1+  \frac{2\al (\bt \sin(2\al y_1) +\al \sinh(2\bt y_2)) }{\al^2 +\bt^2  - \bt^2 \cos(2\al y_1) +\al^2 \cosh (2\bt y_2) }  \Big] .
\eee
\end{lem}

\begin{proof}
See e.g. \cite{AM}.
\end{proof}

The following result is not difficult to prove.

\begin{cor}\label{Cor0}
We have
\be\label{LogL0}
\int_0^x\mathcal M  =  2\bt x + \log(\al^2 +\bt^2  - \bt^2 \cos(2\al y_1) +\al^2 \cosh (2\bt y_2) ) -L_0,
\ee
where
\[
L_0:=  \log(\al^2 +\bt^2  - \bt^2 \cos(2\al x_1) +\al^2 \cosh (2\bt  x_2) ).
\]
\end{cor}

\medskip

\begin{cor}\label{CosT}
Under the assumptions of Lemma \ref{Equal}, we have
\bee
-(\bt-i\al) \int_0^x \cos \Big(\frac{\widetilde B+\widetilde Q}{\sqrt{2}}\Big) &=&   (\bt+i \al) x \\
& &  +\log(\al^2 +\bt^2  - \bt^2 \cos(2\al y_1) +\al^2 \cosh (2\bt  y_2) ) \\
& &  -  \log (1+ e^{2\bt y_2 + 2i\al y_1 })  -L_0 +L_1  , 
\eee
with $L_0$ and $L_1$ defined in \eqref{LogL0} and \eqref{LogL}.
\end{cor}

\begin{proof}
We have from Corollaries \ref{Cor1} and \ref{Cor0}, and  \eqref{LogL},
\bee
&& \int_0^x \cos \Big(\frac{\widetilde B+\widetilde Q}{\sqrt{2}}\Big)  =  x  -  \frac1{(\bt-i\al)} \int_{0}^x (\mathcal M-\mathcal N )\\
&  & \qquad=  x  -  \frac1{\bt-i\al} \Big[  2\bt x + \log(\al^2 +\bt^2  - \bt^2 \cos(2\al y_1) +\al^2 \cosh (2\bt y_2) )   \\
&  &\qquad \qquad-  \log (1+ e^{2(\bt y_2 + i\al y_1) })  -L_0 + L_1 \Big] \\
&  & \qquad =-  \frac1{\bt-i\al} \Big[  (\bt +i\al) x + \log(\al^2 +\bt^2  - \bt^2 \cos(2\al y_1) +\al^2 \cosh (2\bt y_2) )   \\
&  &\qquad \qquad -  \log (1+ e^{2(\bt y_2 + i\al y_1) })  -L_0 + L_1 \Big] , 
\eee
as desired.
\end{proof}

\begin{cor}\label{3p10}
Assume that $x_1,x_2 \in \R$ do not satisfy \eqref{BadCondFF}. Consider the function 
\bea
\mu(x; \al,\bt, x_1,x_2) & := & 2\sqrt{2}\al^2 \bt^2  \frac{\cosh(\bt y_2)\cos(\al y_1) + i \sinh(\bt y_2)\sin(\al y_1)}{\al^2 \cosh^2(\bt y_2) +\bt^2\sin^2(\al y_1)}  \nonu \\
& =& \bt \widetilde B_1 - i \al\widetilde B_2. \label{mu0}
\eea
Then we have
\be\label{infi1}
\lim_{x\to \pm \infty} \mu(x) =0,
\ee
and
\be\label{infi2}
\mu_x = (\bt -i\al)  \cos \Big(\frac{\widetilde B +\widetilde Q}{\sqrt{2}}\Big) \mu.
\ee
\end{cor}

\begin{proof}

Identity \eqref{infi1} is trivial. Let us prove \eqref{infi2}.  First of all, note that (cf. \eqref{tQ12})
\be\label{Qis0}
\bt \widetilde Q_1 - i \al\widetilde Q_2 \equiv 0.
\ee
On the other hand, from  \eqref{zerot} we have
\[
(\widetilde B_1- \widetilde Q_1)_x  - (\bt-i\al) \cos \Big(\frac{\widetilde B+\widetilde Q}{\sqrt{2}}\Big)(\widetilde B_1 + \widetilde Q_1)=0.
\]
Similarly, 
\[
(\widetilde B_2- \widetilde Q_2)_x  - (\bt-i\al) \cos \Big(\frac{\widetilde B+\widetilde Q}{\sqrt{2}}\Big)(\widetilde B_2 + \widetilde Q_2) =0.
\]
We have then
\bee
\mu_x &  =&  ( \bt \widetilde B_1 - i \al\widetilde B_2)_x \\
& =& (\bt-i\al) \cos \Big(\frac{\widetilde B+\widetilde Q}{\sqrt{2}}\Big)\mu \\
& &  + (\bt \widetilde Q_1 - i \al\widetilde Q_2)_x +  (\bt-i\al) \cos \Big(\frac{\widetilde B+\widetilde Q}{\sqrt{2}}\Big)(\bt \widetilde Q_1 -i\al  \widetilde Q_2) \\
& =& (\bt-i\al) \cos \Big(\frac{\widetilde B+\widetilde Q}{\sqrt{2}}\Big)\mu.
\eee
The proof is complete.
\end{proof}

\begin{lem}\label{Nondeg}
Assume that  \eqref{BadCondFF} does not hold. Then $\mu$ defined in \eqref{mu0} has no zeroes, i.e. $|\mu(x)|>0$ for all $x\in \R$.
\end{lem}

\begin{proof}
From \eqref{mu0} we have $\mu(x)=0$ if and only if $\cos(\bt y_1) =0$ and $\sinh(\al y_2) =0$, i.e. from \eqref{shco} we have that \eqref{BadCondFF} is satisfied.  
\end{proof}

Now we consider the opposite case, where the sign in front of \eqref{infi2} is negative. We finish this section with the following result. 

\begin{lem}\label{lemmanu2}
Assume that \eqref{BadCondFF} does not hold. Then 
\[
\mu^1(x;\al,\bt,x_1,x_2) := \frac 1\mu (x;\al,\bt,x_1,x_2),
\]
with $\mu$ defined in \eqref{mu0}, is well-defined, it has no zeroes and satisfies
\[
\lim_{x\to \pm \infty} |\mu^1(x)| = +\infty,
\]
and
\[
\mu^1_x = -(\bt -i\al)  \cos \Big(\frac{\widetilde B +\widetilde Q}{\sqrt{2}}\Big) \mu^1.
\] 
\end{lem}

\begin{proof}
A direct consequence of Corollary \ref{3p10} and Lemma \ref{Nondeg}.
 \end{proof}

\bigskip

\section{Double B\"acklund transformation for mKdV}\label{6}

\medskip

Assume that $x_1$ and $x_2$ do not satisfy \eqref{BadCondFF}. Consider the breather and soliton profiles $B$ and $Q$ defined in \eqref{BBB} and \eqref{Q}, well-defined according to Lemma \ref{Good}.  From Lemma \ref{Equal}, we have the following result.

\begin{lem}\label{BQ0}
We have, for all $x\in \R$,
\[
G(  B, Q,\widetilde B_t, \widetilde Q_t, \bt-i\al) =(0,0).
\]
\end{lem}
Note that the previous identity can be extended by zero to the case where $x_1$ and $x_2$ satisfy \eqref{BadCondFF}, in such a form that now $G(B,Q,\widetilde B_t, \widetilde Q_t, \bt-i\al)$ as a function of $x_1$ and $x_2$ is well-defined and continuous everywhere in $\R^2$ (and identically zero).

\medskip

Define (cf. \eqref{tB}-\eqref{tBt}),
\be\label{B00}
\begin{cases}
\widetilde B^0(x;\al,\bt)  := \widetilde B(x;\al,\bt, 0,0), \\
\widetilde B^0_t (x;\al,\bt) :=  \delta \widetilde B_1(x;\al,\bt, 0,0) +\ga \widetilde B_2(x;\al,\bt, 0,0)  ,\\
B^0 (x;\al,\bt) : = \partial_x \widetilde B(x;\al,\bt, 0,0).
\end{cases}
\ee
Finally, for $z_a^0 \in H^1(\R)$ we define
\be\label{theta0}
\omega_a^0:= -((z_a^0)_{xx} + (z_a^0)^3) \in H^{-1}(\R).
\ee

\medskip

In what follows we will use Lemma \ref{BQ0} and apply Propositions \ref{IFT} and \ref{IFT2} in a neighborhood of the complex soliton and the breather at time zero.  Recall that from Lemma \ref{Good} the complex soliton $Q^0$ is everywhere well-defined since  \eqref{BadCondFF} is not satisfied.

\begin{lem}\label{L1}
There exists $\eta_0>0$ and a constant $C>0$ such that, for all $0<\eta<\eta_0$, the following is satisfied. Assume that $z_a^0\in H^1(\R)$ satisfies
\[
 \|z_a^0\|_{H^1(\R)} <\eta, \quad \omega_a^0 \; \hbox{defined by} \; \eqref{theta0}.
\]
Then there exist unique $z_b^0 \in H^1(\R,\Com)$, $\omega_b^0\in H^{-1}(\R;\Com)$ and $m_1 \in \Com$, of the form 
\[
z_b^0(x)=  z_b^0[z_a^0](x),\quad \omega_b^0(x)=\omega_b^0[z_a^0,\omega_a^0](x) , \quad m_1 =m_1 [z_a^0]:=  \bt- i\al + p^0,
\]
such that 
\[
\|z_b^0\|_{H^1(\R;\Com)} +|p^0| \leq C\eta,
\]

\[
\widetilde z_a+\widetilde z_b \in H^2(\R;\Com),
\]
and
\[
G( B^0 +z_a^0,  Q^0 + z_b^0 ,  \widetilde B^0_t + \omega_a^0 , \widetilde Q^0_t +\omega_b^0,m_1) \equiv (0,0).
\]
\end{lem}

\begin{proof}
Let $Q^0$ and $B^0$ be the soliton and breather profiles defined in \eqref{Q00} and \eqref{B00}.  We will apply Proposition \ref{IFT} with
\[
u_a^0:= B^0, \quad u_b^0:= Q^0,  \quad v_a^0:= \widetilde B_t^0,\quad v_b^0:= \widetilde Q_t^0, \quad m^0:= \bt+i\al. 
\]
Clearly $\re m^0 =\bt>0$, so that \eqref{rem0} is satisfied. On the other hand, \eqref{Conds11} is a consequence of Lemma \ref{BQ0}. From \eqref{zerot} condition \eqref{Conds12} reads
\[
\sin \Big( \frac{\widetilde B^0 + \widetilde Q^0}{\sqrt{2}}\Big) = \frac{(B^0-Q^0)}{\sqrt{2}(\bt -i\al)} \in H^1(\R;\Com).
\]
Condition \eqref{Conds13} is clearly satisfied (see \eqref{tQinfty} and \eqref{tBinfty}). From Corollary \ref{3p10} we have 
\[
\mu^0 = \bt (\widetilde B_1)^0 -i\al (\widetilde B_2)^0.
\]
Note that from Lemmas \ref{Good} and \ref{Nondeg}  $\mu^0$ has no zeroes in the complex plane and it is exponentially decreasing in space.  Finally, let us show that 
\[
\int_\R \mu^0 \sin \Big( \frac{\widetilde B^0 + \widetilde Q^0}{\sqrt{2}} \Big) = \frac{4i\al\bt}{\bt -i\al}.
\]
First of all, we have from \eqref{Qis0}
\bee
 & &  [\bt(\widetilde B_1)^0 -i\al (\widetilde B_2)^0 ] \sin \Big(\frac{\widetilde B^0+\widetilde Q^0}{\sqrt{2}}\Big) \\
 & & \qquad =  [ \bt(\widetilde B_1 + \widetilde Q_1)^0 -i\al (\widetilde B_2 + \widetilde Q_2)^0  ]  \sin \Big(\frac{\widetilde B^0+\widetilde Q^0}{\sqrt{2}}\Big) \\
 & &   \qquad \qquad   +  [  -\bt (\widetilde Q_1)^0 + i\al (\widetilde Q_2)^0 ]  \sin \Big(\frac{\widetilde B^0+\widetilde Q^0}{\sqrt{2}}\Big).
\eee 
Consequently, 
\bee
&&   [\bt(\widetilde B_1)^0 -i\al (\widetilde B_2)^0 ] \sin \Big(\frac{\widetilde B^0+\widetilde Q^0}{\sqrt{2}}\Big) \\
 & & \qquad  =  - \sqrt{2}  \bt \partial_{x_1}\Big[   \cos \Big(\frac{\widetilde B+\widetilde Q}{\sqrt{2}}\Big) \Big]\Big|^0  + i\al  \sqrt{2} \partial_{x_2}\Big[   \cos \Big(\frac{\widetilde B+\widetilde Q}{\sqrt{2}}\Big) \Big] \Big|^0.
\eee 
Therefore, if $R_1,R_2>0$ are independent of $x_1$ and $x_2$,
\bee
& &\int_{-R_2}^{R_1} \mu^0 \sin \Big( \frac{\widetilde B^0 + \widetilde Q^0}{\sqrt{2}} \Big) \\
& & \qquad = \sqrt{2}\int_{-R_2}^{R_1}   \Big\{ -   \bt \partial_{x_1}\Big[   \cos \Big(\frac{\widetilde B+\widetilde Q}{\sqrt{2}}\Big) \Big]\Big|^0  + i\al   \partial_{x_2}\Big[   \cos \Big(\frac{\widetilde B+\widetilde Q}{\sqrt{2}}\Big) \Big] \Big|^0  \Big\}\\
& & \qquad =   \sqrt{2}  \Big\{ -   \bt \partial_{x_1} \int_{-R_2}^{R_1}   \cos \Big(\frac{\widetilde B+\widetilde Q}{\sqrt{2}}\Big) + i\al   \partial_{x_2} \int_{-R_2}^{R_1}  \cos \Big(\frac{\widetilde B+\widetilde Q}{\sqrt{2}}\Big)   \Big\} \Big|^0 .
\eee
Now we use Corollary \ref{CosT}: we have
\bee
& & \partial_{x_1} \int_{-R_2}^{R_1}   \cos \Big(\frac{\widetilde B+\widetilde Q}{\sqrt{2}}\Big) \\
& & \quad = -\frac1{\bt-i\al} \Big[ \frac{2i\al e^{2\bt y_2 + 2i\al y_1}}{1+e^{2\bt y_2 + 2 i \al y_1}} - \frac{2\al\bt^2 \sin (2\al y_1)}{\al^2 + \bt^2 -\bt^2 \cos(2\al y_1) +\al^2 \cosh(2\bt y_2)} \Big]\Big|_{-R_1}^{R_2}.
\eee
We have that
\[
\lim_{R_1,R_2 \to \infty}\partial_{x_1} \int_{-R_2}^{R_1}   \cos \Big(\frac{\widetilde B+\widetilde Q}{\sqrt{2}}\Big) = -\frac{2i\al}{\bt -i\al}.
\]
Similarly,
\bee
& & \partial_{x_2} \int_{-R_2}^{R_1}   \cos \Big(\frac{\widetilde B+\widetilde Q}{\sqrt{2}}\Big) = \\
& & \quad = -\frac1{\bt-i\al} \Big[ \frac{2\bt e^{2\bt y_2 + 2i\al y_1}}{1+e^{2\bt y_2 + 2 i \al y_1}} - \frac{2\al^2 \bt \sinh (2\bt y_2)}{\al^2 + \bt^2 -\bt^2 \cos(2\al y_1) +\al^2 \cosh(2\bt y_2)} \Big]\Big|_{-R_1}^{R_2},
\eee
and
\[
\lim_{R_1,R_2 \to \infty}\partial_{x_2} \int_{-R_2}^{R_1}   \cos \Big(\frac{\widetilde B+\widetilde Q}{\sqrt{2}}\Big) = - \frac{2\bt -4\bt}{\bt -i\al} = \frac{2\bt}{\bt -i\al}.
\]
Adding the previous identities, we finally obtain
\[
\int_\R \mu^0 \sin \Big( \frac{\widetilde B^0 + \widetilde Q^0}{\sqrt{2}} \Big) = \sqrt{2} \Big[ \frac{2i\al\bt}{\bt -i\al}  +  \frac{2i\al \bt}{\bt -i\al}\Big] = \frac{4i\al\bt}{\bt -i\al} \neq 0.
\]
After applying Proposition \ref{IFT}, we conclude.
\end{proof}

Now we address the following very important question: is $y_a^0$ given in Lemma \ref{L2} real-valued for all $x \in \R$? 
In general, it seems that the answer is negative; however, if $z_a^0$ in Lemma \ref{L1} is real-valued, and $z_b^0$ from Lemma \ref{L1} satisfies \eqref{smallNU}, then the corresponding function $y_a^0$ given in Lemma \ref{L2} is also real-valued. This property is a consequence of a deep result called \emph{permutability theorem}, that we explain below.

\medskip

First of all, from Lemma \ref{L1} we have
\be\label{111}
\frac 1{\sqrt{2}} (B^0+z_a^0 -Q^0 -z_b^0) =(\bt-i\al + p^0) \sin\Big( \frac{\widetilde B^0+\widetilde z_a^0 +\widetilde Q^0 +\widetilde{z}_b^0}{\sqrt{2}}\Big), 
\ee
for some small $p^0 \in \Com$, and
\be\label{KK}
\sin\Big( \frac{\widetilde B^0+\widetilde z_a^0 +\widetilde Q^0 +\widetilde{z}_b^0}{\sqrt{2}}\Big) \in H^1(\R;\Com).
\ee
Now, by taking $\eta_0$ smaller if necessary, such that $C\eta< \nu_0$ for all $0<\eta<\eta_0$, Lemma \ref{L2} also applies. We get
\be\label{112}
\frac 1{\sqrt{2}} (Q^0+z_b^0-y_a^0) =(\bt + i\al + q^0) \sin\Big( \frac{\widetilde Q^0+\widetilde{z}_b^0 +\widetilde y_a^0}{\sqrt{2}}\Big),
\ee
for some small $q^0$.

\smallskip

We need some auxiliary notation. Define
\[
\bt_* := \bt + \re p^0, \quad \al_* := \al - \ima p^0,
\]
such that 
\[
\bt  -i\al +p^0 = \bt_* -i\al_*. \qquad \hbox{(Compare with \eqref{ABStar}.)}
\]
We also consider
\[
\widetilde Q^0_* := \widetilde Q(\cdot \ ; -\al_*,\bt_*,0,0), \qquad  Q^0_* :=  Q(\cdot \ ; -\al_*,\bt_*,0,0).
\]
Note that since $p^0$ is small, we have that $Q^0_*$ and $\overline{Q}^0$ share the same properties, since
\be\label{sal}
\|Q^0_* - \overline{Q}^0\|_{H^1(\R;\Com)} \leq C\eta.
\ee
Moreover, thanks to Lemma \ref{Equal0} applied to $Q^0_*$, 
 \[
 \frac 1{\sqrt{2}} Q^0_* =(\bt - i\al +p^0) \sin\Big( \frac{\widetilde Q^0_*}{\sqrt{2}}\Big).
 \]
Consequently, applying Proposition \ref{IFT2} starting at $y_a^0$, and using \eqref{sal} we can define $z_d^0$ via the identity
\be\label{221}
\frac 1{\sqrt{2}} (\overline{Q}^0 +z_d^0 -y_a^0) =(\bt - i\al + p^0) \sin\Big( \frac{\widetilde{\overline{Q}}^0 +\widetilde z_d^0 +\widetilde y_a^0}{\sqrt{2}}\Big). 
\ee

\medskip

Similarly, using \eqref{ABStar} and \eqref{B00} we define
\be\label{B0Star}
(\widetilde B^0)^*:= \widetilde B^0 (\cdot \ ; \al^*,\bt^*), \qquad (B^0)^* := B(\cdot \ ; \al^*,\bt^*),
\ee
so that from Lemma \eqref{Equal} we have
\[
\frac 1{\sqrt{2}} ((B^0)^*  - (Q^0)^*  ) =(\bt^* - i\al^* ) \sin\Big( \frac{(\widetilde{B^0})^*  +(\widetilde{Q^0})^*}{\sqrt{2}}\Big),
\]
and applying Corollary  \ref{Cor0a} we get
\[
\frac 1{\sqrt{2}} ((B^0)^*  - \overline{(Q^0)^* } ) =(\bt + i\al +q^0) \sin\Big( \frac{(\widetilde{B^0})^*  + \overline{ (\widetilde{Q^0})^* }}{\sqrt{2}}\Big).
\]
Using that 
\[
\|(B^0)^* - B^0\|_{H^1(\R)} \leq C\eta, \quad  \|\overline{(Q^0)^* } -  \overline{Q^0}\|_{H^1(\R;\Com)} \leq C\eta,
\]
we can use Proposition \ref{IFT2} to obtain
\be\label{222}
\frac 1{\sqrt{2}} (B^0+z_c^0 - \overline{Q}^0 - z_d^0 ) =(\bt+ i\al + q^0) \sin\Big( \frac{\widetilde B^0 +\widetilde z_c^0 +\widetilde{\overline{Q}^0}+\widetilde z_d^0}{\sqrt{2}}\Big),
\ee
for some $z_c^0$ small. Note that the coefficients $(\bt - i\al + p^0)$ and $(\bt + i\al + q^0)$ were left fixed this time.\footnote{Note that the order of the coefficients will be important.} Note additionally that  $z_d^0$ and $z_c^0$ are bounded functions. Now we can announce a permutability theorem \cite[p. 246]{La}. This is part of a more general result, standard in the mathematical physics literature, see \cite{Wah} for a \emph{formal} proof in the Korteweg-de Vries (KdV) case.

\begin{thm}[Permutability theorem]
We have
\be\label{z0z1}
\widetilde{z_c}^0\equiv \widetilde{z_a}^0.
\ee
In particular, $z_c^0$ is an $H^1$ real-valued function.
\end{thm}

\begin{proof}
Define
\be\label{u012}
u_0:=y_a^0; \quad  u_1:= Q^0+z_b^0; \quad u_2:=  \overline Q^0 +z_d^0;
\ee

\be\label{u1221}
u_{12}:=  B^0 +z_a^0; \quad u_{21}:= B^0 +z_c^0;
\ee
and
\be\label{k1k2}
\kappa_1 := \bt + i\al + q^0, \quad \kappa_2 :=\bt-i\al + p^0.
\ee
Since $p^0$ and $q^0$ are small quantities, we have $\kappa_1\neq \kappa_2$, and both are nonzero complex numbers. Equations \eqref{111}-\eqref{222} read now
\be\label{41}
\frac { (u_1 -u_0)}{\sqrt{2}} = \kappa_1 \sin\Big( \frac{\widetilde u_1 + \widetilde u_0}{\sqrt{2}}\Big) ,
\ee
\be\label{42}
\frac {(u_{12} -u_1)}{\sqrt{2}}  = \kappa_2 \sin\Big( \frac{\widetilde u_{12} + \widetilde u_1}{\sqrt{2}}\Big) ,
\ee
\[
\frac { (u_{2} -u_0)}{\sqrt{2}}  = \kappa_2 \sin\Big( \frac{\widetilde u_2 + \widetilde u_0}{\sqrt{2}}\Big) ,
\]
and
\[
\frac { (u_{21} -u_2)}{\sqrt{2}} = \kappa_1 \sin\Big( \frac{\widetilde u_{21} + \widetilde u_2}{\sqrt{2}}\Big) .
\]
Note that $u_1$ and $u_2$ are obtained via the Implicit Function Theorem and therefore there is an associated uniqueness property for solutions obtained in a small neighborhood of the breather. The idea is to prove that $\widetilde u_{21} \equiv \widetilde u_{12}.$ Define $\widetilde u_3$ via the identity
\be\label{permu2}
\frac{\widetilde u_3 - \widetilde u_1}{2\sqrt{2}} = - \arctan\Big[ \Big(\frac{\kappa_1-\kappa_2}{\kappa_1+\kappa_2}\Big) \tan\Big(\frac{\widetilde u_{12} -\widetilde u_0}{2\sqrt{2}}\Big) \Big].
\ee
Whenever $u_1 = Q^0$, $u_{12} =B^0$, $u_0=0$, $\kappa_1 = \bt + i\al $ and $\kappa_2 =\bt - i\al $, we get from \eqref{breather},
\bee
\frac{\tilde u_3 -\tilde Q^0}{2\sqrt{2}} & =&   -\arctan \Big[ i\frac\al\bt \tan \Big(\frac{\tilde B^0}{2\sqrt{2}} \Big)\Big] = -\arctan \Big( i \frac{\sin(\al x)}{\cosh(\bt x)} \Big) \\
& =& -\arctan \Big(\frac{e^{i\al x} -e^{-i\al x} }{ e^{\bt x} + e^{-\bt x}} \Big).
\eee
Therefore, using \eqref{tQ},
\bee
\tilde u_3 & =&  2\sqrt{2} \arctan( e^{(\bt + i\al)x}) -2\sqrt{2} \arctan \Big(\frac{e^{i\al x} -e^{-i\al x} }{ e^{\bt x} + e^{-\bt x}} \Big) \\
& =&  2\sqrt{2} \arctan \Bigg( \frac{e^{(\bt + i\al)x} - \frac{e^{i\al x} -e^{-i\al x} }{ e^{\bt x} + e^{-\bt x}}}{1 +  e^{(\bt + i\al)x}  . \frac{(e^{i\al x} -e^{-i\al x}) }{ e^{\bt x} + e^{-\bt x}}} \Bigg) \\
& =&  2\sqrt{2} \arctan (e^{(\bt - i\al)x} ) \\
& =&  \overline{\widetilde{Q}^0}.
\eee
Consequently, under the smallness assumptions in \eqref{u012}-\eqref{k1k2} (the open character of these sets is essential) we have that $\tilde u_3$ is still well-defined on the real line with values on the complex plane, and it is close to $ \overline{\widetilde{Q}^0}$, as well as $\widetilde u_2$. 

\medskip

 Let us find an equation for $\widetilde u_3$. As usual, define $u_3 := (\widetilde u_3)_x$. We claim that
\be\label{ecuu3}
\frac {(u_3 -u_0)}{\sqrt{2}}  = \kappa_2 \sin\Big( \frac{\widetilde u_3 + \widetilde u_0}{\sqrt{2}}\Big),
\ee
in other words, $\widetilde u_3\equiv \widetilde u_2$. Similarly, if $\widetilde u_4$ solves
\be\label{permu3}
\frac{\widetilde u_2 - \widetilde u_4 }{2\sqrt{2}} = - \arctan\Big[ \Big(\frac{\kappa_1- \kappa_2}{\kappa_1+ \kappa_2}\Big) \tan\Big(\frac{\widetilde u_{21} -\widetilde u_0}{2\sqrt{2}}\Big) \Big],
\ee
then
\[
\frac {(u_4 -u_0)}{\sqrt{2}}  = \kappa_1 \sin\Big( \frac{\widetilde u_4 + \widetilde u_0}{\sqrt{2}}\Big),
\] 
which implies $\widetilde u_4 \equiv \widetilde u_1$.  Finally, from \eqref{permu2} and \eqref{permu3} we have $\widetilde u_{12} \equiv \widetilde u_{21}$, which proves \eqref{z0z1}. Even better, we have\footnote{Note that this identity is well-defined at one particular set of functions, then extended by continuity.}
\be\label{permu4}
\tan\Big(\frac{\widetilde u_{12} -\widetilde u_0}{2\sqrt{2}}\Big) = -\Big(\frac{\kappa_1+\kappa_2}{\kappa_1-\kappa_2}\Big) \tan\Big(\frac{\widetilde u_2 - \widetilde u_1}{2\sqrt{2}}\Big).
\ee

\medskip

Now let us prove \eqref{ecuu3}. First of all, denote 
\be\label{ele}
\ell:= \frac{\kappa_1+\kappa_2}{\kappa_1-\kappa_2}.
\ee
We have from \eqref{permu2}
\[
\frac{\widetilde u_{12} -\widetilde u_0}{\sqrt{2}} =-2 \arctan \Big[ \ell \tan\Big( \frac{\widetilde u_3 -\widetilde u_1}{2\sqrt{2}}\Big) \Big] ,
\]
so that 
\[
 u_{12} - u_0 = \frac{-\ell (u_3-u_1) \sec^2 \Big(\frac{\widetilde u_3 -\widetilde u_1}{2\sqrt{2}} \Big) }{1+ \ell^2 \tan^2 \Big(\frac{\widetilde u_3 -\widetilde u_1}{2\sqrt{2}} \Big)}.
\]
We also check that 
\[
\sin \Big(\frac{\widetilde u_{12} -\widetilde u_0}{\sqrt{2}} \Big) = \frac{-2\ell \tan \Big( \frac{\widetilde u_3 -\widetilde u_1}{2\sqrt{2}} \Big)}{ 1+ \ell^2 \tan^2 \Big(\frac{\widetilde u_3 -\widetilde u_1}{2\sqrt{2}} \Big)},
\]
and
\[
\cos\Big(\frac{\widetilde u_{12} -\widetilde u_0}{\sqrt{2}} \Big) =  \frac{1-\ell^2  \tan^2 \Big(\frac{\widetilde u_3 -\widetilde u_1}{2\sqrt{2}} \Big) }{1+ \ell^2 \tan^2 \Big(\frac{\widetilde u_3 -\widetilde u_1}{2\sqrt{2}} \Big)}.
\]
Replacing in \eqref{42} we obtain
\bee
-\ell \frac{(u_3 -u_1)}{\sqrt{2}} \sec^2\Big(\frac{\widetilde u_3 -\widetilde u_1}{2\sqrt{2}} \Big)  & =&  \kappa_1 \sin \Big( \frac{\widetilde u_1+\widetilde u_0}{\sqrt{2}} \Big) \Big[ 1+ \ell^2 \tan^2 \Big(\frac{\widetilde u_3 -\widetilde u_1}{2\sqrt{2}} \Big)\Big]\\
& & +  \kappa_2 \sin \Big( \frac{\widetilde u_1+\widetilde u_0}{\sqrt{2}} \Big)\Big[ 1- \ell^2 \tan^2 \Big(\frac{\widetilde u_3 -\widetilde u_1}{2\sqrt{2}} \Big)\Big] \\
& & -2 \ell \kappa_2 \cos\Big( \frac{\widetilde u_1+\widetilde u_0}{\sqrt{2}} \Big) \tan \Big(\frac{\widetilde u_3 -\widetilde u_1}{2\sqrt{2}} \Big).
\eee
Using \eqref{ele} and \eqref{41} we have
\bee
u_3 -u_0 - \sqrt{2}\kappa_1 \sin\Big( \frac{\widetilde u_1+\widetilde u_0}{\sqrt{2}} \Big) & =& - \sqrt{2}\cos^2 \Big(\frac{\widetilde u_3 -\widetilde u_1}{2\sqrt{2}} \Big) \times\\
& & \times  \Big[ (\kappa_1 -\kappa_2) \sin \Big( \frac{\widetilde u_1+\widetilde u_0}{\sqrt{2}} \Big) \Big( 1+ \ell \tan^2 \Big(\frac{\widetilde u_3 -\widetilde u_1}{2\sqrt{2}} \Big)\Big) \\
& & \qquad  -2 \kappa_2 \cos\Big( \frac{\widetilde u_1+\widetilde u_0}{\sqrt{2}} \Big) \tan \Big(\frac{\widetilde u_3 -\widetilde u_1}{2\sqrt{2}} \Big) \Big],
\eee
i.e., after some standard trigonometric simplifications,
\bee
u_3 -u_0 & =& \sqrt{2} \kappa_2 \sin \Big( \frac{\widetilde u_1+\widetilde u_0}{\sqrt{2}} \Big) \cos\Big(\frac{\widetilde u_3 -\widetilde u_1}{\sqrt{2}} \Big)  +  \sqrt{2}\kappa_2 \cos\Big( \frac{\widetilde u_1+\widetilde u_0}{\sqrt{2}} \Big) \sin \Big(\frac{\widetilde u_3 -\widetilde u_1}{2\sqrt{2}} \Big)\\
& =& \sqrt{2} \kappa_2 \sin \Big(\frac{\widetilde u_3 +\widetilde u_0}{\sqrt{2}} \Big),
\eee
as desired.
\end{proof}

Another consequence of the previous result is the following equivalent result. 

\begin{cor}
We have
\[
z_d^0 \equiv \overline{z_b^0} \qquad \hbox{and} \qquad p^0 =\overline{q^0}.
\]
In other words, $\al^* =\al_*$ and $\bt^*=\bt_*.$
\end{cor}

\begin{proof}
Note that $z_a^0 \equiv z_c^0$. From \eqref{222} we have
\[
\frac 1{\sqrt{2}} (B^0+z_a^0 - Q^0 - \overline{z_d^0}) =(\bt- i\al + \overline{ q^0}) \sin\Big( \frac{\widetilde B^0 +\widetilde z_a^0 +\widetilde{Q}^0+\overline{\widetilde{z_d}^0}}{\sqrt{2}}\Big). 
\]
From \eqref{111} and the uniqueness of $z_b^0$ and $p^0$ as implicit functions of $z_a^0$, we conclude.
\end{proof}

The key result of this paper is the following surprising property.

\begin{cor}\label{RealY}
The function $y_a^0$ is real-valued. Moreover, there is a small ball of data $z_a^0$ in $H^1(\R)$ for which the corresponding data $z_b^0$ lies in an open set of  $H^1(\R;\Com)$.
\end{cor}

\begin{proof}
The second statement is a consequence of the Implicit Function Theorem. On the other hand, the first one is consequence of the permutability theorem. First of all, note that 
\be\label{over}
\overline{(\bt+i\al + q^0)} = \bt -i\al + p^0 = \bt^* -i\al^*.
\ee
Now from \eqref{permu4} we get
\[
\tan \Big(\frac{B^0+z_a^0 -y_a^0}{2\sqrt{2}} \Big) = -\frac{(\bt + \re p^0)}{i(\al - \ima p^0)} \tan\Big(\frac{\widetilde Q^0+\widetilde z_b^0 - \widetilde{\overline{Q}^0} - \widetilde{\overline{z_b}^0}}{2\sqrt{2}}\Big),
\]  
namely
\[
\tan \Big(\frac{B^0+z_a^0 -y_a^0}{2\sqrt{2}} \Big) = -\frac{(\bt + \re p^0)}{(\al - \ima p^0)} \tanh\Big(\frac{\ima (\widetilde Q^0+\widetilde z_b^0)}{\sqrt{2}}\Big),
\]
from which we have $y_a^0 (x)$ real-valued for all $x \in\R$. 
\end{proof}

The main advantage of the double B\"acklund transformation is that  now the dynamics of $y_a^0$ is real-valued.  We apply Theorem \ref{T2a} with the initial data $z_b^0$ to get a complex solution of mKdV $u_b(t) = Q^*(t) + z_b(t)$ defined for all $t\neq t_k$ and satisfying \eqref{ComplexStab}.

\medskip

Now we reconstruct $z_a(t)$. As in \eqref{B0Star}, let us define, using \eqref{tB}, \eqref{ABStar} and \eqref{DGStar},
\be\label{dinatB}
\widetilde B^*(t,x):= \widetilde B(x; \al^*,\bt^*, \delta^* t +x_1,\ga^* t+x_2),
\ee
and
\be\label{dinaB}
B^*(t,x) = \partial_x \widetilde B^*(t,x), \quad \widetilde B_j^* (t,x) := \widetilde B_j(x; \al^*,\bt^*, x_1,x_2)\Big|_{x_1=\delta^* t+x_1, \, x_2 =\ga^* t+x_2}.
\ee
In other words, we recover the original breather in \eqref{breather} with scaling parameters $\al^*$ and $\bt^*$ and shifts $x_1,x_2$, provided they do not depend on time. Finally, as in \eqref{tBt} we define
\[
\widetilde B_t^* (t,x):= \delta \widetilde B_1^* (t,x)+ \ga \widetilde B_2^*(t,x).
\]

\begin{lem}\label{44}
Assume that $t \in \R$ is such that \eqref{BadCond} holds. Then there are unique $z_a=z_a(t) \in H^1(\R;\Com) $  and $w_a=w_a(t) \in H^{-1}(\R;\Com)$ such that 

\be\label{H1ab}
\widetilde z_a+\widetilde z_b \in H^2(\R;\Com),
\ee

\be\label{311}
\frac 1{\sqrt{2}} (B^*+z_a -Q^* -z_b ) =(\bt-i\al + p^0) \sin\Big( \frac{\widetilde B^*+\widetilde z_a +\widetilde Q^*+\widetilde z_b}{\sqrt{2}}\Big),
\ee
where $\widetilde B^*$ and $B^*$ are defined in \eqref{dinatB} and \eqref{dinaB}. Moreover, we have
\bea\label{411}
& & 0= \widetilde B_t^*  + w_a - \widetilde Q_t^*  - w_b   \nonu\\
& &  \qquad + \, (\bt - i\al + p^0)  \Big[ ( B_{x}^* +(z_a)_x +  Q_x^* + (z_b)_x) \cos \Big(\frac{\widetilde B^*+ \widetilde z_a +\widetilde Q^*+ \widetilde z_b}{\sqrt{2}} \Big) \nonu\\
& & \qquad  \qquad \qquad \qquad  + \frac { ((B^*+z_a)^2 + (Q^*+z_b)^2)}{\sqrt{2}}\sin \Big( \frac{\widetilde B^*+ \widetilde z_a +\widetilde Q^* + \widetilde z_b}{\sqrt{2}}\Big) \Big] \nonu,\\
& & 
\eea
and for all $t\neq t_k,$
\[
\|z_a(t)\|_{H^1(\R;\Com)} \leq C\eta.
\]
\end{lem}

\begin{proof}
We apply Proposition \ref{IFT2} at the point
\[
X^0:= (B^*, Q^*, \widetilde B^*_t ,  \widetilde Q^*_t, \bt-i\al + p^0),
\]
because a slight variation of Lemma \ref{BQ0} shows that (compare with \eqref{over})
\[
G(B^*, Q^*, \widetilde B^*_t ,  \widetilde Q^*_t, \bt-i\al + p^0) =(0,0).
\] 
Since $p^0$ is small, 
\[
\re \, (\bt-i\al + p^0) >0.
\]
On the other hand, \eqref{Conds12} is a consequence of \eqref{KK}. Similarly, from \eqref{tQinfty} we get \eqref{Conds13} satisfied. Finally, in order to ensure that \eqref{Cond212} is clearly satisfied,  we apply  Corollary \ref{lemmanu2}: we get
\[
\mu^1 = \frac1{\mu^*}, \quad \hbox{ where }\quad \mu^* := \bt^* \widetilde B_1^* -i\al^* \widetilde B_2^*;
\]
see Corollary \ref{3p10} and \eqref{dinaB}. Then we conclude thanks to Proposition \ref{IFT2}.
\end{proof}

\begin{cor}
The function $z_a(t)$ as defined in \eqref{311} is real-valued.
\end{cor}
\begin{proof}
 The same proof as in Corollary \ref{RealY} works \emph{mutatis mutandis}, since now $y_a(t)$ is real-valued.
\end{proof}

\begin{prop}
For all $t\neq t_k$, $u_a= B^*+z_a$ is an $H^1$ real-valued solution to mKdV with initial data $u_0$.  Therefore, by uniqueness,\footnote{Technically, what we need is a result about unconditional uniqueness, however, from \cite{KO} one can conclude that such a result is valid for mKdV on the line if we consider data with $H^1$ regularity. } $B^*+z_a \equiv u$.
\end{prop}

\begin{proof}
Since $u_b=Q^*+z_b$ solves mKdV, we use \eqref{311}-\eqref{411} and Theorem \ref{Cauchy1} to conclude.
\end{proof}

\bigskip

\section{Stability of breathers}\label{7}

\medskip

In this final paragraph we prove Theorem \ref{T1}. In what follows, we assume that $u_0\in H^1(\R)$ satisfies (\ref{In}) for some $\eta$ small. Let $u\in C(\R; H^1(\R))$ be the --unique in a certain sense-- associated solution of the Cauchy problem (\ref{mKdV}), with initial data $u(0)=u_0$. 
Finally, we recall the conserved quantities mass \eqref{M1} and energy \eqref{E1}.

\begin{proof}[\bf Proof of Theorem \ref{T1}]
Consider $\ve_0>0$ small but fixed, and $0<\eta<\eta_0$ small. From Lemmas \ref{L1} and \ref{44} the proof is not difficult. Indeed, define the tubular neighborhood
\bea\label{V}
\mathcal V(A_0,\eta):= \Big\{  U \in H^1(\R) \ | \  \inf_{\tilde x_1,\tilde x_2\in \R} \|U - B(\cdot ; \al,\bt, \tilde x_1,\tilde x_2)\| \leq A_0 \eta   \Big\}.
\eea
Note that $B$ represents here the breather profile defined in \eqref{BBB}. The original breather $B(t)$ from \eqref{breather} can be recovered using \eqref{dinatB} as follows (there is a slight abuse of notation here, but it is easily understood):
\[
B(t,x;\al,\bt,x_1,x_2) = B(x;\al,\bt,\delta t +x_1,\ga t +x_2).
\]
We will prove that if $u(t) \in \mathcal V(A_0,\eta)$ for $t\in [0,T_0]$, with $T_0>0$ and $|T_0- t_k|>\ve_0$, for all $k\in \Z$, then 
\[
u(t) \in\mathcal V(A_0/2,\eta),
\]
which proves the result for all positive times far from the points $t_k$. First of all, by taking $\eta_0>0$ smaller if necessary, and $\eta\in (0,\eta_0)$, we can ensure that there are unique $x_1(t)$, $x_2(t) \in \R$, defined on $[0,T_0]$, and such that
\be\label{Z}
z(t,x):= u(t,x) - B(x;\al,\bt,\delta t +x_1(t),\ga t +x_2(t))
\ee
satisfies
\be\label{O1}
\int_\R z(t,x)B_1(x;\al,\bt,\delta t +x_1(t),\ga t +x_2(t)) dx=0,
\ee
and
\be\label{O2}
\int_\R z(t,x)B_2(x;\al,\bt,\delta t +x_1(t),\ga t +x_2(t)) dx =0.
\ee
The directions $B_1$ and $B_2$ are defined in \eqref{B1}-\eqref{B2} (see \cite{AM} for a similar statement and its proof). Moreover, we have
\[
\|z(0)\|_{H^1(\R)} \lesssim \eta,
\]
and similar estimates for $x_1(0)$ and $x_2(0)$, with constants not depending on $A_0$ large. Therefore condition \eqref{BadCondFF} is not satisfied. For the sake of simplicity, we can assume $x_1(0)=x_2(0)=0$, otherwise we perform a shift in space and time on the solution to set them equal zero. 

\medskip

Define $z_a^0 := z(0)$ and apply Lemma \ref{L1}, and then Lemma \ref{L2} with the corresponding $z_b^0$ obtained from Lemma \ref{L1}. We will obtain a \emph{real-valued} seed $y_a^0$ small in $H^1(\R)$. Note that the constants involved in each inversion do not depend on $A_0$. In particular, the differences between $\al$ and $\al^*$, and $\bt$ and $\bt^*$ are not depending on $A_0$:
\be\label{abab}
|\al -\al^*| +|\bt -\bt^*| \lesssim \eta.
\ee

\medskip

Next, we evolve the mKdV equation with initial data $y_a^0$. From Theorem \ref{KPV} we have the bound \eqref{MaMa} for the dynamics $y_a(t)$. On the other hand, decomposition \eqref{O1}-\eqref{O2} implies that 
\be\label{x1x2}
|x_1'(t)| +|x_2'(t)| \lesssim A_0 \eta,
\ee
from which the set of points where condition \eqref{BadCond} is not satisfied is still a countable set of isolated points (see Lemma \ref{4p2}). 

\medskip

Now we are ready to apply Lemmas \ref{L3} and \ref{44} with parameters $\al^*, \bt^*$ and shifts $x_1(t)$ and $x_2(t)$ in \eqref{dinatQ}, \eqref{dinaQ} and \eqref{dinatB}-\eqref{dinaB}. In that sense, we have chosen a unique set of parameters for each fixed time $t$, and the mKdV solution that we choose is the same as the original $u(t)$. Indeed, just notice that at $t=0,$ we have from \eqref{312} at $t=0$ and \eqref{112}, 
\[
\frac 1{\sqrt{2}} (Q^*(0) + z_b(0)  -y_a^0) =(\bt + i\al + q^0) \sin\Big( \frac{\widetilde Q^*(0)+\widetilde z_b(0)+\widetilde y_a^0}{\sqrt{2}}\Big),
\]
\[
\frac 1{\sqrt{2}} (Q^0+z_b^0-y_a^0) =(\bt + i\al + q^0) \sin\Big( \frac{\widetilde Q^0+\widetilde{z}_b^0 +\widetilde y_a^0}{\sqrt{2}}\Big).
\]
Using the uniqueness of the solution obtained by the Implicit function theorem in a neighborhood of the base point, we have
\be\label{igualdad}
z_b(0) = Q^0-Q^*(0)+z_b^0 \sim z_b^0.
\ee
Now we use \eqref{311} at $t=0$ and \eqref{111}:
\[
\frac 1{\sqrt{2}} (B^*(0)+z_a(0) -Q^*(0) -z_b(0) ) =(\bt-i\al + p^0) \sin\Big( \frac{\widetilde B^*(0)+\widetilde z_a(0) +\widetilde Q^*(0)+\widetilde z_b(0)}{\sqrt{2}}\Big),
\]
and
\[
\frac 1{\sqrt{2}} (B^0+z_a^0 -Q^0 -z_b^0) =(\bt-i\al + p^0) \sin\Big( \frac{\widetilde B^0+\widetilde z_a^0 +\widetilde Q^0 +\widetilde{z}_b^0}{\sqrt{2}}\Big).
\]
From \eqref{igualdad}, we have
\[
\frac 1{\sqrt{2}} (B^*(0)+z_a(0) -Q^0 -z_b^0 ) =(\bt-i\al + p^0) \sin\Big( \frac{\widetilde B^*(0)+\widetilde z_a(0) +\widetilde Q^0+\widetilde z_b^0}{\sqrt{2}}\Big).
\]
Once again, since $B^0$ and $B^*(0)$ are close, using the uniqueness of the solution obtained via the Implicit function Theorem, we conclude that 
\[
B^*(0)+z_a(0) =B^0+z_a^0.
\]
Since both initial data are the same, we conclude that the solution obtained via the B\"acklund transformation is $u(t)$.

\medskip

Note that the constants involved in the inversions are not  depending on $A_0$. We finally get 
\be\label{zeroA}
\sup_{|t-t_k| \geq \ve_0}\big\| u(t) - B^*(t) \big\|_{H^1(\R)}\leq C_0 \eta,
\ee
where 
\[
B^*(t,x):= B(x;\al^*,\bt^*,\delta^* t +x_1(t),\ga^* t +x_2(t)).
\]
Finally, from \eqref{abab} and after redefining the shift parameters, we get the desired conclusion, since for $A_0$ large enough, we have $C_0 \leq \frac 12 A_0$.

\medskip

Now we deal with the remaining case $t\sim t_k$. Fix $k\in \Z$. Note that $z_a = u -B^*$ satisfies the equation
\be\label{equZ}
(z_a)_t + [(z_{a})_{xx} + 3(B^*)^2 z_a +  3 B^* z_a^2 + z_a^3]_x  + x_1'(t)B_1^* + x_2'(t) B_2^* =0,
\ee
in the $H^1$-sense. In what follows, we will prove that, maybe taken $\ve_0$ smaller but independent of $k$, we have 
\be\label{Boots1}
\sup_{|t-t_k| \leq \ve_0}\big\| u(t) - B^*(t) \big\|_{H^1(\R)}\leq 4 A_0 \eta.
\ee
Since $A_0$ grows with $\ve_0$ small, that implies that, after choosing $\eta_0$ smaller if necessary, such an operation can be performed without any risk.

\medskip

In what follows, we assume that there is $T^* \in ( t_k -\ve_0, t_k+\ve_0]$ such that, for all $t\in [t_k -\ve_0, T^*] $,
\be\label{4A0}
\big\| z_a(t) \big\|_{H^1(\R)}\leq 4 A_0 \eta,
\ee
and $T^*$ is maximal in the sense of the above definition (i.e., there is no $T^{**}>T^*$ satisfying the previous property). If $T^* =  t_k+\ve_0$, there is nothing to prove and   \eqref{Boots1} holds.

\medskip

Assume $T^*< t_k+\ve_0$. Now we consider the quantity
\[
\frac 12\int_\R z_a^2(t), \quad t\in [t_0-\ve_0, T^*].
\]
We have after \eqref{equZ},
\bee
\partial_t \frac 12\int_\R z_a^2(t) & =&   \int_\R (z_a)_x \Big[ 3(B^*)^2 z_a + 3B^* z_a^2 +z_a^3\Big](t)\\
& &  +x_1'(t) \int_\R z_a(t) B_1^* + x_2'(t) \int_\R z_a(t) B_2^*.
\eee
Using \eqref{4A0} and \eqref{x1x2}, we have for some --explicit-- fixed constant $C>0$ depending only on $\al,\bt$ and $4$, and $\eta_0$ even smaller if necessary,
\[
\abs{\partial_t \frac 12\int_\R z_a^2(t) } \leq CA_0^2 \eta^2.
\] 
After integration in time and using \eqref{zeroA}, we have
\[
\int_\R z_a^2(T^*) \leq \int_\R z_a^2(t_0-\ve_0) + C \ve_0A_0^2 \eta^2 \leq 1.9 A_0^2 \eta^2,
\]
if $\ve_0$ is small but fixed. A similar estimate can be obtained for $\|(z_a)_x(t)\|_{H^1(\R)}$ by proving an estimate of the form
\[
\abs{\partial_t \frac 12\int_\R (z_a)_x^2(t)} \leq  CA_0^2 \eta^2.
\]
Therefore estimate \eqref{4A0} has been bootstrapped, which implies that $T^* =t_0+\ve_0.$  Note that the estimates do not depend on $k$, but only on the length of the intervals $\sim \ve_0$.\footnote{Note that an argument involving the uniform continuity of the mKdV flow will not work in this particular case since the sequence of times $(t_k)$ is unbounded.}

\medskip

We conclude that there is $\tilde A_0>0$ fixed such that 
\[
\sup_{t\in\R }\big\| u(t) - B^*(t) \big\|_{H^1(\R)}\leq \tilde A_0 \eta.
\]

\medskip

Finally, estimates \eqref{Fn1} and \eqref{Fn2} are obtained from \eqref{x1x2}, and using the fact that $\al^*$ and $\bt^*$ are close to $\al$ and $\bt$ in terms of $C\eta$.  The proof is complete.
\end{proof}

\begin{rem}
From the proof and the results in \cite{CKSTT} it is easy to realize that the evolution of breathers can be estimated in a polynomial form in time for any $s>\frac 14$, however, in order to make things simpler, we will not address this issue.
\end{rem}

\begin{cor}
We have for all $t\neq t_k$
\bee
 \frac 12 \int_\R (B^*+z_a)^2(t) & =&  \frac 12 \int_\R (Q^*+z_b)^2(t) + 2(\bt^* -i\al^*) \\
&  = &  M[y_a^0]  + 4\bt^* .
\eee
Moreover, this identity can be extended to any $t\in \R.$
\end{cor}

\begin{proof}
Same as Corollary \ref{CLaw0}.
\end{proof}

Finally, we recall that $\ga^* = 3(\al^*)^2 -(\bt^*)^2$ and $E[u] = \frac 12 \int_\R u_x^2 - \frac 14\int_\R u^4.$

\begin{cor}
Assume that $t\neq t_k$  for all $k\in \Z$. Then we have
\bee
E[B^* + z_a](t)  & =&  E[Q^* + z_b](t)   - \frac43 (\bt^* - i\al^*)^3 \\
 &=&   E[y_a^0]  +\frac 43 \bt^* \ga^* .
\eee
Finally, this quantity can be extended in a continuous form to every $t\in \R.$ 
\end{cor}

\begin{proof}
Same as Corollary \ref{CLaw3}.
\end{proof}


\bigskip

\section{Asymptotic Stability}\label{8}

\medskip

{\bf We finally prove Theorem  \ref{T3}}. Note that for some $c_0>0$ depending on $\eta>0,$
\be\label{Asy0}
\lim_{t\to +\infty} \|y_a(t)\|_{H^1(x\geq c_0 t)} =0.
\ee
This result can be obtained by adapting the proof for the soliton case in the Martel and Merle's paper \cite{MMnon}.  Indeed, consider 
\[
\phi(x):= \frac K\pi \arctan (e^{x/K}), \quad K>0,
\]
so that 
\be\label{properties}
\lim_{-\infty} \phi =0, \quad \lim_{+\infty} \phi =1, \quad  \phi''' \leq \frac1{K^2} \phi', \quad \phi'>0 \hbox{  on } \R.
\ee
Fix $c_0, t_0>0$. Consider the quantities
\[
I(t):= \frac 12\int_\R y_a^2 (t) \phi(x-c_0 t_0 + \frac 12c_0 (t_0 -t)),
\] 
\[
J(t):= \int_\R \Big[  \frac 12 (y_a)_x^2 (t) -\frac 14 y_a^4 (t) + \frac 12 y_a^2 (t)\Big] \phi(x-c_0 t_0 + \frac 12c_0 (t_0 -t)).
\] 
It is not difficult to see that 
\[
I'(t) = -\frac 14c_0 \int_\R y_a^2 \phi' (t)+ \frac 12 \int_\R y_a^2 \phi''' (t) -\frac 32 \int_\R (y_a)_x^2 \phi'(t) +\frac 34 \int_\R y_a^4 \phi'(t),
\]
so that using \eqref{properties}, and if $c_0>0$ is small (but depending on $\eta$ smaller if necessary),
\[
I'(t) \leq 0.
\]
We have then
\[
I(t_0) \leq I(0)= \frac 12\int_\R y_a^2 (0) \phi(x- c_0 t_0),
\]
and
\[
\lim_{t\to +\infty} I(t) =0.
\]
A similar result holds for $J(t)$, which proves \eqref{Asy0}.

\medskip

Note that  $\widetilde z_b + \widetilde y_a \in H^2(\R;\Com)$ (see \eqref{H1ba}). In what follows, we will prove that this function satisfies \emph{better estimates} than $y_a$ and $z_b$ if $x$ is taken large. 

\medskip

Fix $t\neq t_k$ large, with $|t - t_k|\geq \ve_0$. We use the notation 
\be\label{trick}
\widetilde z_c: =\widetilde y_a+ \widetilde z_b.
\ee
Note that from \eqref{Conds234} we have 
\[
\|\widetilde z_c(t) \|_{H^2(\R;\Com)} \leq C\nu,
\]
with $C=C(\ve_0)$ independent of time.
From the B\"acklund transformation \eqref{312} we obtain
\bee
(\widetilde z_c)_x  - 2y_a &  = &  \sqrt{2}(\bt + i\al + q^0) \Big[ \sin\Big( \frac{\widetilde Q^*+\widetilde z_c}{\sqrt{2}}\Big) - \sin\Big( \frac{\widetilde Q^*}{\sqrt{2}}\Big) \Big] \\
&  = &  \sqrt{2}(\bt + i\al + q^0) \Big[ \sin\Big( \frac{\widetilde Q^*}{\sqrt{2}}\Big) \Big\{  \cos\Big( \frac{\widetilde z_c}{\sqrt{2}}\Big) -1 \Big\}  + \sin\Big( \frac{\widetilde z_c}{\sqrt{2}}\Big)\cos\Big( \frac{\widetilde Q^*}{\sqrt{2}}\Big) \Big] \\
& =&  Q^* \Big\{  \cos\Big( \frac{\widetilde z_c}{\sqrt{2}}\Big) -1 \Big\}  + \sqrt{2} \sin\Big( \frac{\widetilde z_c}{\sqrt{2}}\Big) \frac{Q^*_x}{Q^*}.
\eee
Assume now that $x>c_0 t/2$. Then we have for some fixed constant $c>0$,
\[
\abs{  \frac{Q^*_x}{Q^*}  + m} \leq e^{-c x}, \qquad m=\bt + i \al + q^0 = \bt^* + i\al^*,
\]
and
\[
(\widetilde z_c)_x + m \widetilde z_c = g, 
\]
where 
\[
g:= Q^* \Big\{  \cos\Big( \frac{\widetilde z_c}{\sqrt{2}}\Big) -1 \Big\}  + \sqrt{2} \Big\{ \sin\Big( \frac{\widetilde z_c}{\sqrt{2}}\Big) -  \frac{\widetilde z_c}{\sqrt{2}} \Big\} \frac{Q^*_x}{Q^*} +   \widetilde z_c \Big\{ \frac{Q^*_x}{Q^*}  + m \Big\} + 2y_a.
\]
Solving the previous ODE we get
\[
\widetilde z_c(t,x) = \widetilde z_c(t,c_0 t/2) e^{-m(x-c_0 t/2)} + \int_{c_0t/2}^x g(t,s) e^{-m(x-s)}ds,
\]
so that
\[
|\widetilde z_c(t,x)| \lesssim  |\widetilde z_c(t,c_0 t/2)| e^{-\bt^*(x-c_0 t/2)} + \int_{c_0t/2}^x |g(t,s)| e^{-\bt^*(x-s)}ds.
\]
From the Young's inequality we get
\[
\norm{\widetilde z_c (t)}_{L^2(x\geq c_0 t)} \lesssim |\widetilde z_c(t,c_0 t/2)| e^{-\bt^*c_0 t/2} + \|g(t)\|_{L^2(x\geq c_0 t)} e^{-\bt^*c_0 t} .
\]
Clearly
\[
|\widetilde z_c(t,c_0 t/2)| \lesssim \|\widetilde z_c (t)\|_{H^1(\R;\Com)} \leq C\nu, \quad  \|g(t)\|_{L^2(x\geq c_0 t)} \leq C\nu^2 +C\nu e^{-c t}+o(1).
\]
Passing to the limit, we obtain for all $T_n \to +\infty$, $|T_n - t_k|\geq \ve_0$ for all $n$ and $k$,
\[
\lim_{n \to +\infty}\norm{\widetilde z_c(T_n) }_{L^2(x\geq c_0 T_n)}  =0.
\]
A similar result can be obtained in the case of $z_c$ and $(z_c)_x$.  Finally, from \eqref{trick} we get  
\be\label{Asy1}
\lim_{n \to +\infty}\norm{z_b (T_n)}_{H^1(x\geq c_0 T_n)}  =0.
\ee

\medskip

Finally, we repeat the same strategy with \eqref{311} and \eqref{H1ab} to obtain
\[
\lim_{t\to +\infty} \|z_a(T_n)\|_{H^1(x\geq c_0 T_n)} =0.
\]
Note that since the flow map is continuous in time with values in $H^1$, we can extend the result to any sequence $T_n \to +\infty$ by choosing an $\ve_0>0$ smaller but still independent of $k$.

\medskip

\noindent
{\bf Proof of Corollary \ref{Ine}}. Assume that for all $c_0>0$ small it is possible to find $\tilde T_0>0$ very large such that 
\[
 \|z_a(\tilde T_0)\|_{H^1(\R)} < c_0\|z^0_a\|_{H^1(\R)}. 
\]
With no loss of generality we can assume $\tilde T_0 \neq t_k$ for all $k$, otherwise we perturb $\tilde T_0$ and the previous inequality still holds true. We apply the B\"acklund transform twice to find
\[
 \|y_a(\tilde T_0)\|_{H^1(\R)} < Cc_0\|z^0_a\|_{H^1(\R)},
\]
and therefore
\[
 \|y_a(0)\|_{H^1(\R)} < Cc_0\|z^0_a\|_{H^1(\R)},
\]
for some constant $C>0.$ Two consecutive inversions of the B\"acklund transform leads to 
\[
0< \|z_a^0\|_{H^1(\R)} < Cc_0\|z^0_a\|_{H^1(\R)},
\]
which is a contradiction if $c_0>0$ is chosen small enough.

\bigskip
\bigskip

\appendix

\section{Proof of  Lemma \ref{Equal}}\label{AppA}

\medskip

We will use the specific character of the breather and soliton profiles. Since \eqref{BadCondFF} does hold, both $\widetilde Q$ and $Q$ are well-defined everywhere. We have
\[
\sin \Big(\frac{\widetilde B+\widetilde Q}{\sqrt{2}}\Big)   =   \sin (2(\arctan \Theta_1 +\arctan \Theta_2)), 
\]
where from \eqref{tQ} and \eqref{tB},
\[
\Theta_2:=e^{ \bt y_2 + i \al y_1 }, \quad \Theta_{1}:=\frac{\bt}{\al}\frac{\sin(\al y_1)}{\cosh(\bt y_2)}.
\]
We have
\bee
\sin \Big(\frac{\widetilde B+\widetilde Q}{\sqrt{2}}\Big) & =&  2 \Big[\sin (\arctan \Theta_1)\cos(\arctan \Theta_2)+\sin(\arctan \Theta_2)\cos(\arctan \Theta_1) \Big] \\
&  & \quad  \times  \Big[\cos (\arctan \Theta_1)\cos(\arctan \Theta_2)-\sin(\arctan \Theta_1)\sin(\arctan \Theta_2) \Big]\\
& =&  2\Big[\tan (\arctan \Theta_1)\cos^2(\arctan \Theta_1)\cos^2 (\arctan \Theta_2) \\
& &  \qquad  -\sin^2 (\arctan \Theta_1) \tan(\arctan \Theta_2)\cos^2(\arctan \Theta_2) \\
& & \qquad +\cos^2 (\arctan \Theta_1) \tan (\arctan \Theta_2)\cos^2(\arctan \Theta_2) \\
& & \qquad -\sin^2 (\arctan \Theta_2)\tan (\arctan \Theta_1)  \cos^2(\arctan \Theta_1) \Big].
\eee
Since
\[
\sin^2(\arctan(z))=\frac{z^2}{1+z^2}, \quad \cos^2(\arctan(z))=\frac{1}{1+z^2},
\]
we have
\be\label{Tri1}
\sin \Big(\frac{\widetilde B+\widetilde Q}{\sqrt{2}}\Big) = \frac{2(\Theta_1 -\Theta_1^2 \Theta_2 + \Theta_2 -\Theta_2^2\Theta_1)}{(1+\Theta_1^2)(1+\Theta_2^2)} .
\ee
On the other hand
\bee
\frac{1}{\sqrt{2}} ( B - Q)  &  = &   2\partial_x(\arctan \Theta_1 -\arctan \Theta_2) \\
& = & 2 \Big(\frac{\Theta_{1,x}}{1+\Theta_1^2} - \frac{\Theta_{2,x}}{1+\Theta_2^2} \Big) \\
& =&2 \frac{(1+\Theta_2^2) \Theta_{1,x} - (1+\Theta_1^2)\Theta_{2,x}}{(1+\Theta_1^2)(1+\Theta_2^2)}.
\eee
Hence, collecting terms and factorizing, from \eqref{zerot} we are lead to prove that 
\be\label{Tri2}
(1+\Theta_2^2)\Theta_{1,x} - (1+\Theta_1^2)\Theta_{2,x}- (\beta-i\alpha)(\Theta_1 -\Theta_1^2 \Theta_2 + \Theta_2 -\Theta_2^2\Theta_1) =0.
\ee
Now we perform some computations. We have from  \eqref{tQ},
\be\label{T2x}
\Theta_{2,x}=(\beta+i\alpha)\Theta_2,
\ee

\be\label{T2xx}
\al  ( \bt+ i\al \Theta_1^2 ) \cosh^2(\bt y_2) =  \bt (\al \cosh^2(\bt y_2) + i \bt \sin^2(\al y_1) ),
\ee
and
\[
\Theta_{1,x}  =  \Big(\frac{\bt}{\al}\frac{\sin(\al y_1)}{\cosh(\bt y_2)} \Big)_x=\frac{\al\bt\cos(\al y_1)\cosh(\bt y_2)-\bt^2\sin(\al y_1)\sinh(\bt y_2)}{\al\cosh^2(\bt y_2)} ,
\]
so that 
\be\label{T1x}
\Theta_{1,x} - (\bt -i\al)\Theta_1 =\bt \Big[ \frac{\al e^{i\al y_1} \cosh(\bt y_2) -\bt e^{\bt y_2} \sin(\al y_1) }{\al \cosh^2(\bt y_2)} \Big],
\ee
and
\bea\label{T1x2}
[\Theta_{1,x} + (\bt -i\al)\Theta_1]\Theta_2^2 & =&  \bt \Big[ \frac{\al e^{-i\al y_1} \cosh(\bt y_2) + \bt e^{-\bt y_2} \sin(\al y_1) }{\al \cosh^2(\bt y_2)} \Big]e^{2(\bt y_2 + i\al y_1)} \nonu\\
& =&  \bt \Theta_2 \Big[ \frac{\al e^{\bt y_2} \cosh(\bt y_2) + \bt e^{i\al y_1} \sin(\al y_1) }{\al \cosh^2(\bt y_2)} \Big].
\eea
%
%
Using \eqref{T2x}, \eqref{T2xx}, \eqref{T1x} and \eqref{T1x2} we have
\bee
& & \hbox{l.h.s. of \eqref{Tri2}}  = \\
& & =   (1+\Theta_2^2)\Theta_{1,x} - 2(\bt + i\al \Theta_1^2)\Theta_{2}- (\beta-i\alpha)(1  -\Theta_2^2)\Theta_1 \\
& & = [\Theta_{1,x} -(\bt -i\al) \Theta_1]  +[\Theta_{1,x} + (\bt -i\al) \Theta_1] \Theta_2^2  - 2(\bt + i\al \Theta_1^2)\Theta_{2}\\
& & = \bt \Big[ \frac{\al e^{i\al y_1} \cosh(\bt y_2) -\bt e^{\bt y_2} \sin(\al y_1) }{\al \cosh^2(\bt y_2)} \Big]+ \\
& & \quad +\bt \Theta_2 \Big[ \frac{ \al e^{\bt y_2} \cosh(\bt y_2) + \bt e^{i\al y_1} \sin(\al y_1)   -2 \al \cosh^2(\bt y_2) -2 i \bt\sin^2(\al y_1) }{\al \cosh^2(\bt y_2)} \Big]\\
& & = \bt \Big[ \frac{\al e^{i\al y_1} \cosh(\bt y_2) -\bt e^{\bt y_2} \sin(\al y_1) }{\al \cosh^2(\bt y_2)} \Big]+ \\
& & \quad +\bt \Theta_2 \Big[ \frac{ -\al e^{-\bt y_2} \cosh(\bt y_2) + \bt e^{-i\al y_1} \sin(\al y_1) }{\al \cosh^2(\bt y_2)} \Big] \\
& & = 0,
\eee
which proves \eqref{Tri2}.
\bigskip

\end{document}